\documentclass[10pt,electronic]{amsart}   

\usepackage[applemac]{inputenc}
\usepackage[T1]{fontenc}
\usepackage[small,it,margin=2pt]{caption}
\usepackage[english,french]{babel}
\usepackage[hmargin=3.9cm,vmargin=3.5cm]{geometry} 
\usepackage{graphicx}                         
\usepackage{amsmath}
\usepackage{amsthm}                      
\usepackage{amssymb}
\usepackage{tikz-cd}     
\usepackage{tikz}

\usetikzlibrary{matrix,arrows}
\usetikzlibrary{shapes.geometric}
\usepgflibrary[arrows]
\usepgflibrary{arrows}
\usetikzlibrary[arrows]
\usetikzlibrary{calc}

\usepackage{euscript}
\renewcommand{\mathcal}{\EuScript}
\usepackage{mathrsfs}

\usepackage[isbn=false,doi=false,eprint=false,url=false,style=authoryear,dashed=false,backend=bibtex,maxnames=99]{biblatex}
\makeatletter
\newrobustcmd*{\parentexttrack}[1]{%
	\begingroup
	\blx@blxinit
	\blx@setsfcodes
	\blx@bibopenparen#1\blx@bibcloseparen
	\endgroup}
\AtEveryCite{%
	}
\makeatother
\renewcommand{\cite}{\parencite}
\DeclareNameAlias{sortname}{last-first}
\renewbibmacro*{volume+number+eid}{%
	\printfield{volume}%
	\setunit*{\addnbspace}
	\printfield{number}%
	\setunit{\addcomma\space}%
	\printfield{eid}}
\DeclareFieldFormat[article]{volume}{\textbf{#1}}
\DeclareFieldFormat[article]{number}{{\mkbibparens{#1}},}
\DeclareFieldFormat{pages}{#1 pages}
\DeclareFieldFormat
[article,inbook,incollection,inproceedings,patent,thesis,unpublished]
{pages}{#1}
\renewbibmacro{in:}{}
\theoremstyle{plain}                                                           
\newtheorem{thm}{Theorem}[section]
\newtheorem{lem}[thm]{Lemma}
\newtheorem{prop}[thm]{Proposition}

\newtheorem{cor}[thm]{Corollary}

\theoremstyle{definition}

\newtheorem{rem}[thm]{Remark}
\newtheorem{defn}[thm]{Definition}

\DeclareMathOperator{\id}{id}
\newcommand{\field}[1]{\ensuremath{\mathbf{#1}}}
\newcommand{\Q}{\ensuremath{\field{Q}}}        
\newcommand{\C}{\ensuremath{\field{C}}}

\newcommand{\Z}{\ensuremath{\field{Z}}} 

\newcommand{\R}{\field{R}}
\newcommand{\M}{{M}}
\newcommand{\MM}{\overline{{M}}}
\newcommand{\A}{\mathbb{A}}
\renewcommand{\P}{\mathbb P}
\newcommand{\Bl}{\mathrm{Bl}}

\renewcommand{\a}{\mathscr A}
\newcommand{\E}{\mathcal E}
\newcommand{\CC}{\mathcal C}
\newcommand{\BM}{\mathrm{BM}}
\newcommand{\Hycom }{\mathsf{Hycom}}
\newcommand{\chaincoGrav}{\mathfrak{coGrav}}
\newcommand{\chainPrim}{\mathfrak{Prim}}
\newcommand{\Prim}{\mathsf{Prim}}
\newcommand{\chainHycom}{\mathfrak{Hycom}}
\newcommand{\Grav}{\mathsf{Grav}}
\newcommand{\coGrav}{\mathsf{coGrav}}

\newcommand{\barp}{\mathrm{B}^{\mathrm{pl}}}
\newcommand{\barc}{\mathrm{B}^{\mathrm{cyc}}}
\newcommand{\cobarp}{\mathrm{\Omega}^{\mathrm{pl}}}
\newcommand{\cobarc}{\mathrm{\Omega}^{\mathrm{cyc}}}

\newcommand{\freepl}{\mathrm{T}^{\mathrm{pl}}}
\newcommand{\freeantipl}{\mathrm{T}^{\mathrm{pl},-}}

\newcommand{\freecyc}{\mathrm{T}^{\mathrm{cyc}}}
\newcommand{\freeanticyc}{\mathrm{T}^{\mathrm{cyc},-}}

\tikzset{
  graphnode/.style = {align=center, inner sep=0pt, scale=0.3, text centered,
    font=\sffamily},
  vi/.style = {graphnode, circle, white, font=\sffamily\bfseries, draw=black,
    fill=black, text width=1em},
  ve/.style = {graphnode, circle, draw=black, 
    text width=1em, thick},
  vx/.style = {graphnode, draw=white,
    minimum width=0em, minimum height=0em},
  vb/.style = {graphnode, diamond, white, draw=black, 
    minimum width=1em, minimum height=1em, thick},
} 

\title{Brown's dihedral moduli space and freedom of the gravity operad}
\author{Johan Alm} 
\email{alm@math.su.se}
\address{Matematiska institutionen \\
Stockholms universitet \\ 106 91 Stockholm \\ Sweden }

\author{Dan Petersen}
\email{danpete@math.ku.dk}
\address{Institut for matematiske fag\\
	K{\o}benhavns universitet \\
	Universitetsparken 5 \\
	2100 K{\o}benhavn {\O} \\
	Denmark }

\begin{document} 
 \maketitle   
 \selectlanguage{english}
 \begin{abstract}
Francis Brown introduced a partial compactification $M_{0,n}^\delta$ of the moduli space $M_{0,n}$. We prove that the {gravity cooperad}, given by the degree-shifted cohomologies of the spaces $M_{0,n}$, is cofree as a nonsymmetric anticyclic cooperad; moreover, the cogenerators are given by the cohomology groups of $M_{0,n}^\delta$. As part of the proof we construct an explicit diagrammatically defined basis of $H^\bullet(M_{0,n})$ which is compatible with cooperadic cocomposition, and such that a subset forms a basis of $H^\bullet(M_{0,n}^\delta)$. We show that our results are equivalent to the claim that $H^k(\M_{0,n}^\delta)$ has a pure Hodge structure of weight $2k$ for all $k$, and we conclude our paper by giving an independent and completely different proof of this fact. The latter proof uses a new and explicit iterative construction of $\M_{0,n}^\delta$ from $\A^{n-3}$ by blow-ups and removing divisors, analogous to Kapranov's and Keel's constructions of $\MM_{0,n}$ from $\P^{n-3}$ and $(\P^1)^{n-3}$, respectively.   \end{abstract}

%

\section*{Introduction}

Let $\M_{0,n}$ for $n \geq 3$ be the moduli scheme of $n$ distinct ordered points on $\P^1$ up to the action of $\mathrm{PGL}(2)$, and $\MM_{0,n}$ its Deligne--Mumford compactification. These are smooth affine (resp.\ projective) varieties over $\Q$ (or $\Z$) of dimension $(n-3)$. Motivated by the study of multiple zeta values, Brown introduced an intermediate space $\M_{0,n} \subset \M_{0,n}^\delta \subset \MM_{0,n}$, depending on a dihedral structure $\delta$ on the set $\{1,\ldots,n\}$; that is, an identification with the integers from $1$ to $n$ with the edges of some unoriented $n$-gon. The space $M_{0,n}^\delta$ is again affine, and the union of all spaces $\M_{0,n}^\delta$ over all dihedral structures constitutes an open affine covering of the scheme $\MM_{0,n}$. In more detail, let \(X^{\delta}_n\subset \MM_{0,n}(\R)\) be the closure of the cell parametrizing $n$ distinct points on the circle $\P^1(\R)$, ordered compatibly with the chosen dihedral structure \(\delta\). Then \(\M_{0,n}^{\delta}\) is the subvariety of \(\MM_{0,n}\) formed by adding to \(\M_{0,n}\) only those boundary divisors that have nonempty intersection with \(X_n^{\delta}\).

The relevance of \(\M_{0,n}^{\delta}\) in the theory of periods and multiple zeta values resides on the following. By Grothendieck's theorem on algebraic de Rham cohomology, the cohomology of $\M_{0,n}$ can be computed using the global sections of the complex of algebraic differential forms. It is thus interesting to study integrals of the form $$\int_{X_n^\delta} \omega$$
where $[\omega]$ is any top degree cohomology class. Such integrals typically diverge, since the form $\omega$ may have poles along the boundary of $X_n^\delta$; the integral converges precisely when $[\omega]$ is in the image of the restriction map $H^{n-3}(\M_{0,n}^\delta) \to H^{n-3}(M_{0,n})$. Brown proved that any relative period integral of \(\M_{0,n}\) (in the sense of Goncharov and Manin) can be decomposed as a \(\Q[2i\pi]\)-linear combination of integrals of this form, with $\omega$ defined over $\Q$. Moreover, each such integral evaluates to a rational linear combination of multiple zeta values. The cohomology groups \(H^k(\M_{0,n}^{\delta})\) and their Hodge structures are thus relevant to our understanding of motives and periods.

The degree-shifted cohomologies \(\{H^{\bullet-1}(\M_{0,n})\}_{n \geq 3}\) constitute an (anti)cyclic cooperad with Poincar\'e residue as cocomposition. This cooperad was introduced by Getzler, who called it the \emph{gravity cooperad}, and we denote it \(\coGrav\). The homologies \(\{H_{\bullet}(\MM_{0,n})\}_{n \geq 3}\) constitute a cyclic operad with composition given, simply, by the maps induced by inclusions of boundary strata of \(\MM_{0,n}\). This operad is known as the \emph{hypercommutative operad}, \(\Hycom\), and features prominently in Gromov--Witten theory. Ginzburg, Kapranov and Getzler have shown that the two are interchanged by Koszul duality: in particular, there is a quasi-isomorphism \(\cobarc \coGrav \to \Hycom\) between the cyclic cobar construction on the gravity cooperad and the hypercommutative operad. The statement is, in a sense, encoded by the geometry of \(\MM_{0,n}\). The set of complex points decomposes as a union
 \[
 \MM_{0,n}(\C) = \coprod_{T\in\mathrm{Tree}_n} \prod_{v\in\mathrm{Vert}(T)} \M_{0,n(v)}(\C) 
 \]
of strata labeled by trees. This decomposition says that \(\{\MM_{0,n}(\C)\}_{n \geq 3}\) is the free cyclic operad of sets generated by the collection \(\{\M_{0,n}(\C)\}_{n \geq 3}\) of points of the open moduli spaces. Once we include topology and go from sets to varieties it is no longer a free operad; instead the decomposition is (morally speaking) transformed into said Koszul duality relation.

Brown's partial compactification has a similar structure:
 \[
 \M_{0,n}^{\delta}(\C) = \coprod_{T\in\mathrm{PTree}_n} \prod_{v\in\mathrm{Vert}(T)} \M_{0,n(v)}(\C) 
 \]
is now a union over strata indexed by \emph{planar} trees, which can be read as saying that \(\{\M_{0,n}^{\delta}(\C)\}_{n \geq 3}\) is the \emph{planar} operad of sets freely generated by the collection \(\{\M_{0,n}(\C)\}_{n \geq 3}\). What we here term a planar operad might also be called a nonsymmetric cyclic operad. We call them planar because they are encoded by the combinatorics of planar (non-rooted) trees, just like cyclic operads are encoded by trees, operads by rooted trees, and nonsymmetric operads by planar rooted trees. 

Note that, by Poincar\'e duality we could equally well take the hypercommutative operad as defined by \(\Hycom _n = H^{\bullet-2}(\MM_{0,n})\), with Gysin maps as composition. Analogously, the collection \(\Prim_n = H^{\bullet-2}(\M_{0,n}^{\delta})\) is an operad. Our first statement says that \(\coGrav\) and \(\Prim\) satisfy a duality relation of planar (co)operads, analogous to the duality relation of cyclic (co)operads between \(\coGrav\) and \(\Hycom \).

\begin{thm}\label{thmone}
The planar cobar construction \(\cobarp\coGrav\) and \(\Prim\) are quasi-isomorphic as planar operads if and only if the mixed Hodge structure on \(H^k(\M_{0,n}^{\delta})\) is pure of weight \(2k\). Moreover, the compositions of \(\Prim\) are all zero, so either condition is equivalent to the statement that \(\coGrav\) is (noncanonically) isomorphic to the cofree cooperad cogenerated by \(\Prim\) (with degree shifted by one).
\end{thm}

We remark that we throughout write ``cofree cooperad'' for what should properly be called ``cofree conilpotent cooperad''; we assume all cooperads to be conilpotent.

In the second and third parts of the paper we give independent proofs of the two properties mentioned. This may be logically redundant (the properties are, after all, equivalent), but we believe the proofs to be of independent interest. The second part is devoted to proving the following:

\begin{thm}\label{thmtwo}
The gravity operad is the linear hull of a free nonsymmetric operad of sets.
\end{thm}

Thus $\coGrav$, the linear dual of the gravity operad, is cofree on the (dual of the) linear hull of the generators of said nonsymmetric operad of sets. That we have to switch to the gravity \emph{operad} at this point (and this point only) is an unfortunate minor hiccup, but it is necessary: $\coGrav$ is conilpotent, so it could not possibly be the linear hull of any kind of cooperad of sets. On the other hand we want to compute with differential forms and residues throughout, and the arguments are naturally formulated in terms of the cohomology of $\M_{0,n}$. Thus working with the gravity operad rather than $\coGrav$ throughout would have been somewhat cumbersome.

In any case, this implies that \(\cobarp\coGrav\) and \(\Prim\) are quasi-isomorphic as planar operads, but can also be regarded as showing something stronger. In particular, the result involves construction of an explicit basis \(\{\alpha_G\}\) of \(H^{\bullet}(M_{0,n})\), with a subset \(\{\alpha_P\}\subset\{\alpha_G\}\) forming a basis for the image of \(H^{\bullet}(\M^{\delta}_{0,n})\) in \(H^{\bullet}(\M_{0,n})\). 

In the third and final part we give a direct proof of:

\begin{thm}\label{thmthree}
The mixed Hodge structure on \(H^k(\M_{0,n}^{\delta})\) is pure of weight \(2k\).
\end{thm}

The proof relies on an inductive construction of $\M_{0,n}^\delta$ from $\A^{n-3}$, alternating between blowing up a smooth subvariety and then removing the strict transform of a divisor containing the blow-up center. It is inspired by Hassett's work on moduli spaces of weighted pointed stable curves. This construction of \(\M_{0,n}^{\delta}\) is new.

Our results have several interesting consequences. Let us begin with a rather immediate one:

\begin{cor}[Bergstr\"om--Brown] \label{bb}The ordinary generating functions for the Poincar\'e polynomials of $M_{0,n}$ and $M_{0,n}^\delta$ are compositional inverses of each other. 
\end{cor}

Indeed, this result was proven in \cite{bergstrombrown}, assuming that $H^k(M_{0,n}^\delta)$ is pure of weight $2k$; thus our result fills a gap in their argument. We refer to their paper for a more precise statement of Corollary \ref{bb}. The result gives in particular a simple recursive procedure for computing the Betti numbers of $M_{0,n}^\delta$. In brief, the point is that $M_{0,n}^\delta$ has a stratification all of whose strata are products of moduli spaces of the form $\M_{0,n_i}$, and it is easy to express $\M_{0,n_i}$ (say, in the Grothendieck ring of varieties) as a polynomial in the class of the affine line. Thus also $[M_{0,n}^\delta] = p([\A^1])$ for some polynomial $p$, and Theorem \ref{thmthree} implies that its coefficients record the Betti numbers of $\M_{0,n}^\delta$. 

Another corollary is the deduction of an explicit left inverse to the restriction \(H^{\bullet}(\M_{0,n})\to H^{\bullet}(\M_{0,n}^{\delta})\), which is compatible with operadic structure. This gives a recipe for how to regularize any possibly divergent integral over \(X_n^{\delta}\) of a form \(\omega\in H^{n-3}(M_{0,n})\) in a coherent way, by first projecting \(\omega\) to \(H^{n-3}(\M_{0,n}^{\delta})\). This is used by the first author \cite{Alm} to prove:

\begin{cor}
There is a nontrivial universal \(A_{\infty}\) structure \(\{\nu_n\}_{n\geq 2}\) on Batalin-Vilkovisky algebras, such that the coefficients of \(\nu_n\) are multiple zeta values of weight at most \((n-2)\) and, a priori, any multiple zeta value occurs in the structure.
\end{cor}

While this paper was in the final stages of preparation, the preprint \cite{dupontvallette} appeared on the arXiv, whose results overlap significantly with ours. They, too, show that cofreedom of the gravity cooperad is essentially equivalent to purity of the mixed Hodge structure on $H^\bullet(M_{0,n}^\delta)$, and that (once one has proven cofreedom or purity) the cogenerators of $\coGrav$ will be given by the cohomology groups of $M_{0,n}^\delta$. However, the actual proof of cofreedom they give is completely different. In particular, our proof is constructive, in the sense that we write down an explicit isomorphism between $\coGrav$ and a cofree cooperad defined in terms of diagrams.

\subsection*{Outline of the paper}This paper is divided into three parts which are more or less logically independent, and the reader is invited to begin reading whichever one she finds most interesting.

The first part, Section \ref{sec:duality}, is primarily devoted to the proof of Theorem \ref{thmone}. We first recall the Koszul duality result of Getzler and Ginzburg--Kapranov, that there is a quasi-isomorphism of cyclic operads
 \[
\cobarc \coGrav \to \Hycom .
 \] After sketching Getzler's proof of this theorem we explain an alternative approach to this result. This involves constructing suitable (co)chain versions of both (co)operads, $\chaincoGrav$ and $\chainHycom$, and proving that we have an \emph{iso}morphism
 \[
\cobarc \chaincoGrav \cong \chainHycom.
 \]
Then one can use Hodge theory to deduce that both (co)chain (co)operads are \emph{formal}, which proves the result in a slightly different way. We then discuss how this has an analogue in the planar (nonsymmetric cyclic) case: we construct a chain operad \(\chainPrim\), with cohomology \(\Prim_n = H^{\bullet-2}(\M_{0,n}^{\delta})\) and obtain an isomorphism
 \[
\cobarp \chaincoGrav \cong \chainPrim
 \]
of planar operads. Theorem \ref{thmone} is proved by arguments based on this isomorphism.

The second part, Section \ref{sec:freedom}, is devoted to a proof of Theorem \ref{thmtwo}. We construct an explicit basis $\{\alpha_G\}$ of $H^\bullet(M_{0,n})$ defined in terms of certain diagrams of chords on a polygon, which we call \emph{gravity chord diagrams}. This basis is compatible with the operadic composition in the nonsymmetric gravity operad, and this makes the set of gravity chord diagrams into an operad in the category of sets. From the combinatorial description of gravity chord diagrams it becomes immediate that this operad is in fact freely generated by a subset of \emph{prime} chord diagrams, which implies Theorem \ref{thmtwo}. 

In the third part, Section \ref{sec:purity}, we define and study a generalization \(\M_{0,\a}^\delta\) of \(\M_{0,n}^{\delta}\), where the points are not just labeled but \emph{weighted}. These are analogues of Hassett's notion of \emph{weighted} stable pointed curves studied in \cite{hassettweighted}. Hassett's curves are parametrized by a moduli space $\MM_{0,\a}$ depending on a ``weight vector'' $\a$ which assigns a weight to each marked point. The usual space $\MM_{0,n}$ is recovered when all points have weight $1$; in general one gets new birational models of the moduli space. 
We then introduce  subspaces $\M_{0,\a}^\delta \subset \MM_{0,\a}$, which generalize $\M_{0,n}^\delta \subset \MM_{0,n}$. 

The spaces $\MM_{0,\a}$ satisfy wall-crossing with respect to the weight vector $\a$: when $\a$ crosses a wall, $\MM_{0,\a}$ is modified in a predictable way by a blow-up. This was used by Hassett to recover descriptions of $\MM_{0,n}$ as an iterated blow-up of $\P^{n-3}$ (resp.\ $(\P^1)^{n-3}$) originally due to Kapranov (resp.\ Keel). We prove that when crossing a wall, the moduli space $\M_{0,\a}^\delta$ is modified by blowing up a smooth subvariety and then removing the strict transform of a divisor containing the blow-up center. In this way we obtain an inductive construction of $\M_{0,n}^\delta$ from $\A^{n-3}$. The heart of the proof of Theorem \ref{thmthree} consists in showing that the property of having $H^k$ pure of weight $2k$ is preserved at each step of this inductive construction.

The concluding Appendix \ref{sec:operads} collects terminology regarding operads. There are many flavors of operad in the literature, often described in terms of grafting together graphs of some sort. The usual theory of operads arises when the graphs are rooted trees, but there are also \emph{cyclic operads}, which correspond to trees without a root (the moduli spaces $\MM_{0,n}$ give an example of a cyclic operad), and \emph{nonsymmetric operads}, which correspond to rooted trees with an embedding into the plane up to isotopy (equivalently, a ribbon graph structure). We wish to consider ``nonsymmetric cyclic'' operads, corresponding to unrooted trees with a planar embedding; we propose to call these \emph{planar operads}. We define the notions of planar and antiplanar operad and cooperad, and we define the bar and cobar constructions which act on them. This material will be routine to the experts, but we have included it for completeness since the notion of planar operad seems to have received very little attention in the literature previously. 

\subsection*{Acknowledgements} We would like to thank Cl\'ement Dupont for valuable discussions.

\section{Duality}\label{sec:duality}

We begin this section by recalling the duality between the anticyclic cooperad $\coGrav$ and the cyclic operad $\Hycom $. We recall this in some detail since we will later generalize the argument to the planar case. As we explain, one can either prove duality on the level of cohomology or on the level of chains. The argument for duality on the level of cohomology is in some sense much easier; however, one needs to use that $\coGrav$ and $\Hycom $ carry \emph{pure} Hodge structures of certain weights. 

We then study the analogous relationship between $\coGrav$, considered as an antiplanar cooperad, and the planar operad $\Prim$ given by the cohomology of the spaces $\M_{0,n}^\delta$, with operadic composition given by the pushforward (Gysin) maps in cohomology. We will prove later (Section \ref{sec:purity}) that $H^k(\M_{0,n}^\delta)$ carries a pure Hodge structure of weight $2k$, and using this result one could prove that $\coGrav$ and $\Prim$ are duals of each other under planar bar-cobar duality. However, we believe that it will clarify the logic of the paper to instead prove a chain level statement in this section, namely that $\coGrav$ and a certain dg operad $\chainPrim$ such that $H(\chainPrim) \cong \Prim$ are planar duals of each other. This can be proved without knowing anything about weights.

Using this duality, we then show that purity of $H^\bullet(\M_{0,n}^\delta)$, formality of the operad $\chainPrim$, and cofreedom of $\coGrav$ are all equivalent to each other. 

\subsection{The main operads}
For terminology and conventions regarding operads, see Appendix \ref{sec:operads}. Our terminology should be familiar, except for the notion of \emph{planar operads}. These are what one might also term nonsymmetric cyclic operads, i.e., they have no action of general permutations, only of cyclic groups. Just like there is a forgetful functor from ordinary operads to nonsymmetric operads, there is a forgetful functor from cyclic operads to planar operads (we simply forget the permutation action, but retain the action of cyclic groups). Each flavor of operads has a machinery of free operads and bar and cobar constructions: they only thing that changes is which notion of trees the constructions employ.

We begin with some recollections on the moduli space \(\M_{0,n}\) and its Deligne-Mumford compactification \(\MM_{0,n}\). The open moduli space of smooth genus zero curves with \(n\) labeled points is
 \[
 \M_{0,n} = ((\mathbb{P}^1)^n\setminus diagonals)/\mathrm{PGL}_2,
 \]
where \(\mathrm{PGL}_2\) is the algebraic group of automorphisms of \(\mathbb{P}^1\) and acts by M\"obius transformations. The compactification \(\MM_{0,n}\) is a smooth projective variety. The complement $\MM_{0,n} \setminus \M_{0,n}$ is a strict normal crossing divisor, and as such there is an induced stratification of $\MM_{0,n}$ whose closed strata are the intersections of boundary divisors. This stratification is called the \emph{stratification by topological type}, and can be equivalently defined by declaring that two points lie in the same stratum if and only if the corresponding pointed curves are homeomorphic (working over $\C$). 

The strata in the space $\MM_{0,n}$ are usually indexed by stable dual graphs $\Gamma$ with $n$ external half-edges (legs), see e.g.\ \cite[Section 1]{arbarellocornalba}. One associates to a stable $n$-pointed curve a graph whose vertices correspond to irreducible components, whose edges correspond to nodes, and whose legs correspond to the markings. Thus, in the genus zero case all graphs are actually trees. The stratum corresponding to a tree $T$ can be written as $\prod_{v \in \mathrm{Vert}(T)}\M_{0,n(v)}$ and its closure as $\prod_{v \in \mathrm{Vert}(T)}\MM_{0,n(v)}$, where $n(v)$ denotes the number of half-edges adjacent to a vertex. Stability of the graph means that $n(v) > 2$ for all $v$. It follows that the collection \(\{\MM_{0,n}\}_{n\geq 3}\) is a cyclic operad in the category of projective varieties: The composition
 \[
\circ^j_i: \MM_{0,m+1}\times \MM_{0,n+1} \to \MM_{0,m+n}
 \]
is simply inclusion of a closed stratum. By functoriality the homology \(\{H_\bullet(\MM_{0,n})\}_{n\geq 3}\) is a cyclic operad, too. It is usually denoted \(\Hycom \) and called the hypercommutative operad. However, since we want to have all dg operads with a cohomological grading (and we will later consider a chain level model of $\Hycom $), it will be more convenient to make the ``Poincar\'e dual'' definition that $\Hycom $ is the operad $\{H^{\bullet-2}(\MM_{0,n})\}_{n \geq 3}$, with composition maps given by the Gysin maps (the maps which are Poincar\'e dual to the pushforwards in homology).

The open moduli space is not an operad. However, we can consider $\M_{0,n+1} \times \M_{0,n'+1}$ as a codimension one open stratum inside $\MM_{0,n+n'}$, adjacent to the open stratum $\M_{0,n+n'}$, and get a Poincar\'e residue
 \[
H^{\bullet}(\M_{0,m+n}) \to H^{\bullet-1}(\M_{0,n+1} \times \M_{0,n'+1}).
 \]

\begin{prop}[Getzler]
	The collection of suspensions \(\{H^{\bullet-1}(\M_{0,n})\}_{n \geq 3}\) is an anticyclic cooperad, with cocomposition given by the Poincar\'e residue just defined.  
\end{prop} 
 
 This cooperad is called the \emph{gravity cooperad} and denoted \(\mathsf{coGrav}\). We note that taking Poincar\'e residue is only coassociative up to a sign, which is treated carefully in \cite[Section 3.1]{hodge2}; it is this sign factor that causes the gravity cooperad to be \emph{anti}cyclic.

As it stands, both $\coGrav$ and $\Hycom $ carry natural mixed Hodge structures, but in neither (co)operad are the (co)composition maps compatible with this mixed Hodge structure. To remedy this we need to introduce a Tate twist. We thus let 
$$ \coGrav_n = H^{\bullet-1}(M_{0,n}) \otimes \Q(-1) $$
and 
$$ \Hycom_n = H^{\bullet-2}(\MM_{0,n}) \otimes \Q(-1). $$

\begin{prop}\label{purityprop} The degree $k$ component of $\Hycom_n$ has a pure Hodge structure of weight $k$, and the degree $k$ component of $\coGrav_n$ has a pure Hodge structure of weight $2k$. 
\end{prop}

\begin{proof}The first statement is clear: since $\MM_{0,n}$ is smooth and projective, $H^k(\MM_{0,n})$ is pure of weight $k$. Taking into account the degree-shift and the Tate twist, the conclusion follows. 
	
	For the second statement, we note that the space $ M_{0,n}$ is isomorphic to the complement of an arrangement of hyperplanes in $\A^{n-3}$. It follows from this that the mixed Hodge structure on $H^k(M_{0,n})$ is pure of weight $2k$. Indeed, the cohomology ring of such a complement is generated by the differential forms $d \log(f)$, where $f$ is the defining equation for one of the hyperplanes \cite{brieskorn}; thus the generators in $H^1$ manifestly have Hodge type $(1,1)$. Alternatively, see the short proof in \cite{shapirohyperplane}. 
\end{proof}

To avoid cluttering the notation, we omit these Tate twists in most of what follows, but they will be quite relevant. 

\subsection{Cyclic bar-cobar duality between $\coGrav$ and $\Hycom $, take one}In this subsection we recall a result proven in \cite[Theorem 3.4.11]{ginzburgkapranov} and \cite{getzler94}, that the gravity cooperad is Koszul dual to the hypercommutative operad. To prove this theorem one must  in one way or another use the purity results of Proposition \ref{purityprop}. We will outline a proof of this theorem, following the original approaches of Getzler and Ginzburg--Kapranov.

 However, we will then proceed to take a second approach, which is to work instead on the chain level. There exist suitable dg (co)operads $\chaincoGrav$ and $\chainHycom$ whose cohomologies are $\coGrav$ and $\Hycom$, respectively, and which one can prove are duals of each other without knowing anything about purity of the weights of either (co)operad. The construction of suitable such chain models $\chaincoGrav$ and $\chainHycom$ is somewhat subtle, which makes this kind of argument more involved, but it has the advantage of working also in more general situations where purity fails. We follow here the approach taken in \cite[Proposition 6.11]{getzlerkapranov}, which uses substantially results in the theory of residues and currents. After this, we shall explain that the purity results of  Proposition \ref{purityprop} imply \emph{formality} of both dg operads  $\chaincoGrav$ and $\chainHycom$, which will then give a somewhat different proof of the following theorem.

\begin{thm}[Getzler, Ginzburg--Kapranov]\label{koszul1} Let $\barc$ denote the cyclic bar construction of a cyclic operad, and $\cobarc$ the cyclic cobar construction. There are quasi-isomorphisms 
	$$ \cobarc \coGrav \simeq \Hycom  \qquad \text{and}\qquad \barc \Hycom  \simeq \coGrav.$$
	In other words, $\Hycom$ and $\coGrav$ are Koszul duals to each other.
\end{thm}

\begin{proof}Note that either of the two quasi-isomorphisms in the theorem implies the other, by bar-cobar-duality; we prove the first one.
	
We consider $\MM_{0,n}$ as a filtered space: if $X_i$ denotes the union of all strata of dimension $\leq i$ in the stratification of $\MM_{0,n}$ by topological type, then we get a filtration
$$\emptyset = X_{-1} \subset X_0 \subset X_1 \subset \ldots \subset X_{n-3} = \MM_{0,n}. $$ 
Associated to this filtration is a homology spectral sequence
$$ E^1_{pq} = H_{p+q}(X_p,X_{p-1}) \cong H^{p-q}(X_p \setminus X_{p-1}) \implies H_{p+q}(\MM_{0,n}).$$
The isomorphism in the above is a version of Lefschetz duality. Lefschetz duality may at first not seem to apply, since $X_p$ is not a manifold; however, it is enough that $X_p \setminus X_{p-1}$ is a manifold, since we may resolve the singularities of $X_p$ by blowing up without affecting the relative homology group $H_{p+q}(X_p,X_{p-1})$. We have
$$ X_p \setminus X_{p-1} \cong \coprod_T \prod_{v \in \mathrm{Vert}(T)} M_{0,n(v)}$$ 
where $T$ ranges over all trees with $n$ legs and exactly $(n-2-p)$ vertices. The $E^1$-differential maps the summand corresponding to $T$ to those trees $T'$ from which $T$ can be obtained by contracting a single edge, and the corresponding degree $-1$ map from the cohomology of $\prod_{v \in \mathrm{Vert}(T)} M_{0,n(v)}$ to the cohomology of $\prod_{v \in \mathrm{Vert}(T')} M_{0,n(v)}$ is exactly the Poincar\'e residue; that is, the cocomposition map in the gravity cooperad. 

All in all, this shows that the $E^1$ page of the above spectral sequence may be identified (up to some reindexing) with the cobar construction $\cobarc \coGrav$. Moreover, $E^1_{pq}$ is nonzero only in the region $0 \leq q \leq p$, which implies the existence of an edge map $E^1_{p,p} \to H_{2p}(\MM_{0,n})$. One can verify that this edge map is compatible with operadic composition, therefore giving a map  of dg operads $\cobarc \coGrav \to \Hycom$. 

To show that this map is a quasi-isomorphism, we need to prove that the spectral sequence degenerates at $E^2$, and that $E^2_{p,q}$ is nonzero only for $p=q$. This is where we need to use purity. Keeping track of  the Tate twist by $\Q(p)$ in the isomorphism $H_{p+q}(X_p,X_{p-1}) \cong H^{p-q}(X_p \setminus X_{p-1})$, Proposition \ref{purityprop} implies that $E^1_{pq}$ is pure of weight $-2q$ (since the weights on $E^1$ are determined by the weights of $H^\bullet(M_{0,n}$), and on the other hand that $E^{\infty}_{pq}$ is pure of weight $-p-q$ (since the weights on $E^\infty$ are determined by the weights of $H^\bullet(\MM_{0,n}$), for all $p$ and $q$. It follows that only the $E^1$-differential can be compatible with weights, and that the classes that can survive to $E^\infty$ must be  concentrated along the diagonal for weight reasons. 
\end{proof}

\begin{rem} One can give a ``dual'' proof of this theorem, by instead computing the cohomology of $\M_{0,n}$ using the Leray spectral sequence of the open embedding $M_{0,n} \hookrightarrow \MM_{0,n}$. The Leray spectral sequence for the inclusion of the complement of a strict normal crossing divisor has an easily described $E_2$ page: the entries are given by the cohomologies of all possible intersections of divisors. In our case, a $q$-fold intersection of boundary divisors has the form $\prod_{v \in \mathrm{Vert}(T)} M_{0,n(v)}$, where $T$ is a stable tree with $n$ legs and $(q+1)$ vertices. One finds that
	$$ E_2^{pq} = \bigoplus_T H^p\left(\prod_{v \in \mathrm{Vert}(T)} \MM_{n(v)}\right) \implies H^{p+q}(\M_{0,n}) $$
	where $T$ ranges over trees with $(q+1)$ vertices as above. The $E_2$-differential here is given by the Gysin map for the inclusion of a $(q+1)$-fold intersection of boundary divisor into a $q$-fold intersection. Since the Gysin map is how we defined the composition maps in $\Hycom$, this means that after some reindexing of the above $E_2$ page we obtain exactly the bar construction $\barc \Hycom $. Again there is an edge homomorphism giving a map   between $\barc \Hycom $ and $\coGrav$, and an almost exactly identical weight argument shows that the spectral sequence degenerates at $E_3$ and that the map is a quasi-isomorphism.
\end{rem}

\subsection{Forms and currents with log singularities along a divisor}\label{formscurrents}
We briefly recall material explained in more detail in \cite[pp. 96--97]{getzlerkapranov}, to which we refer for definitions and references for assertions made below. Let us first remind the reader about Borel--Moore homology. If $M$ is a smooth manifold, then the Borel--Moore homology of $M$ is  computed by the complex $\CC_\bullet(M)$ of \emph{currents} on $M$. If $M$ is oriented of dimension $d$, then there is a cap product isomorphism $H^\bullet(M) \cong H_{d-\bullet}^\BM(M)$. If $M$ is a smooth algebraic variety of complex dimension $d$, then the cap product isomorphism is compatible with mixed Hodge structure up to a Tate twist $\Q(d)$. For instance, if $H^i(M)$ is pure of weight $i$ for all $i$, then $H_i^\BM(M)$ is pure of weight $-i$, and if $H^i(M)$ is pure of weight $2i$ for all $i$, then $H_i^\BM(M)$ is pure of weight $2(d-i)$.

Let $X$ be a $d$-dimensional complex manifold and $D \subset X$ a strict normal crossing divisor. We denote by  $\E^\bullet(X,D)$ the complex of $\CC^\infty$ differential forms on $X$ with logarithmic singularities along the divisor $D$. This is a complex of nuclear Fr\'echet spaces, and it computes the cohomology of $X \setminus D$: 
$$ H(\E^\bullet(X,D)) \cong H^\bullet(X\setminus D, \C).$$
Getzler and Kapranov introduce another complex $\CC_\bullet(X,D)$, also of nuclear Fr\'echet spaces, which they call ``de Rham currents with logarithmic singularities along $D$''. All these complexes (for varying $D$) are quasi-isomorphic to each other: there is an inclusion $\CC_\bullet(X,D) \hookrightarrow \CC_\bullet(X)$ into the usual complex of currents, which is a quasi-isomorphism. In particular,
$$ H(\CC_\bullet(X,D)) \cong H_\bullet^{\BM}(X,\C). $$
The complexes $\E^\bullet$ satisfy a K\"unneth formula: if $Y$ is another complex manifold with normal crossing divisor $E$, then 
$$ \E^\bullet(X\times Y, D \times Y \cup X \times E ) \cong \E^\bullet(X,D) \widehat \otimes \E^\bullet(Y,E)$$
where $\widehat{\otimes}$ denotes the projective tensor product.

If $D=D_1 \cup D_2 \cup \ldots \cup D_k$, let $D_I = \bigcap_{i \in I} D_i$ for $I \subseteq \{1,\ldots,k\}$; in particular, $D_\emptyset = X$. Then each $D_I$ is itself a complex manifold, and $$D_I' = D_I \cap \bigcup_{j \notin I}D_j$$
is a strict normal crossing divisor on $D_I$. The Poincar\'e residue defines a map $\E^\bullet(X,D) \to \E^{\bullet-1}(D_i,D_i')$ for all $1 \leq i \leq k$, and more generally for every $l$ a map: $$ \bigoplus_{\vert I \vert = l }\E^\bullet(D_I,D_I') \to \bigoplus_{\vert I \vert = l+1}\E^{\bullet-1}(D_I,D_I').$$
There is then an isomorphism 
\begin{equation}
 \CC_{\bullet}(X,D) \cong \bigoplus_{I \subseteq \{1,\ldots,k\}} \E^{2(d-\vert I \vert) - \bullet}(D_I,D_I') \label{currentiso}
\end{equation}
where on the right hand side we mean the total complex of the double complex whose vertical differential is given by the internal differentials in the complexes  $\E^\bullet(D_I,D_I')$, and whose horizontal differential is given by the Poincar\'e residue. For any divisor $D_i$ there is a map $\CC_{\bullet}(D_i,D_i') \to \CC_\bullet(X,D)$ given by pushforward of currents. The same isomorphism for the divisor $D_i$ reads
$$ \CC_\bullet(D_i,D_i') \cong \bigoplus_{I \subseteq \{1,\ldots,k\} \setminus \{i\}} \E^{2(d-1-\vert I \vert) - \bullet}(D_i,D_i'),$$
which is a subcomplex of the right hand side of \eqref{currentiso}. Under these isomorphisms, pushforward of currents becomes identified with the inclusion of this subcomplex.

When $D$ is empty the isomorphism \eqref{currentiso} says that
$ \CC_\bullet(X) \simeq \E^{2d-\bullet}(X),$
so it implements the cap-product isomorphism between the cohomology and Borel--Moore homology of an oriented manifold.

\subsection{Chain level duality}

Let $D^n = \MM_{0,n} \setminus M_{0,n}$. Let $\chainHycom$ denote the cyclic operad $\{\CC_{2n-4-\bullet}(\MM_{0,n}, D^n)\}$ in the symmetric monoidal category of cochain complexes of nuclear Fr\'echet spaces with projective tensor product. Specifically, $\{\CC_{2n-4-\bullet}(\MM_{0,n}, D^n)\}$ is the double suspension of the operad $\{\CC_{-\bullet}(\MM_{0,n},D^n)\}$, whose operad structure is given by pushforward of currents. Let also $\chaincoGrav$ denote the anticyclic cooperad $\{\E^{\bullet-1}(\MM_{0,n},D^n)\}$ in the same category, whose cooperad structure  is given by taking the residue along a divisor. The cohomologies of $\chainHycom$ and $\chaincoGrav$ are $\Hycom $ and $\coGrav$, respectively. 

\begin{thm} \label{gk} We have an isomorphism $\cobarc \chaincoGrav \cong \chainHycom$ of cyclic operads of dg nuclear Fr\'echet spaces.
\end{thm}

\begin{proof}
	In arity $n$, $\cobarc \chaincoGrav$ is given by $$ \bigoplus_T \widehat{\bigotimes_{v \in \mathrm{Vert}(T)}} \E^{\bullet-2}(\MM_{0,n(v)},D^{n(v)})$$
	with the sum ranging over trees with $n$ legs. If we decompose the divisor $D^n \subset \MM_{0,n}$ into irreducible components,  $D_1 \cup \ldots \cup D_k$, then trees as above with $q+1$ vertices correspond to intersections of $q$ distinct components of $D^n$. Thus we may rewrite $\cobarc \chaincoGrav_n$ as
	$$ \bigoplus_{I \subseteq \{1,\ldots,k\}} \E^{\bullet - 2\vert I \vert - 2 } (D_I,D_I') \cong \CC_{2n-4-\bullet}(\MM_{0,n},D^n)$$
	using the K\"unneth formula for the complexes $\E^\bullet$ and  the isomorphism \eqref{currentiso}. 
	Thus $\chainHycom_n \cong \cobarc \chaincoGrav_n$, which is in fact an isomorphism of cyclic operads.
\end{proof}

An advantage of this argument is that it will work identically to prove an analogous statement between $\chaincoGrav$ considered as a \emph{antiplanar} operad and a planar operad $\chainPrim$ built out of the cohomology of $M_{0,n}^\delta$, as we shall see shortly. 

\subsection{The space $\M_{0,n}^\delta$ and the operad $\Prim$}\label{deltadef}

 From now on we shall once and for all let $\delta$ denote the dihedral structure given by the standard (cyclic) ordering on the set \(\{1,\dots,n\}\). In the introduction we defined $\M_{0,n}^\delta$ as the union inside $\MM_{0,n}$ of all strata meeting the closure of the particular component \(X_n^{\delta}\) of $\M_{0,n}(\R)$. One can give somewhat more explicit alternative descriptions, that are easily seen to be equivalent to each other:

\begin{enumerate}
	\item 
	Recall that boundary divisors in $\MM_{0,n}$ correspond to partitions of $\{1,\ldots,n\}$ into subsets $S \sqcup S'$, both of them with at least two elements. We define $\M_{0,n}^\delta$ to be the complement of all boundary divisors such that $S$ (and $S'$) fail to be intervals with respect to the cyclic order on \(\{1,\dots,n\}\).
	\item Let $\Gamma$ be the dual graph of an $n$-pointed nodal curve. We say that $\Gamma$ is \emph{compatible} with a given dihedral structure on $\{1,\ldots,n\}$ if it can be embedded in the plane such that the induced dihedral structure on the set of legs coincides with the given one. We define $\M_{0,n}^\delta$ to be the open subset of $\MM_{0,n}$ given by strata whose corresponding dual graphs are compatible with the dihedral structure.
	\item Every collection of pairwise non-crossing chords in an $n$-gon gives rise to a tiling of the $n$-gon and  a ``Poincar\'e dual'' graph, as illustrated in Figures \ref{chords1} and \ref{chords2}. We define $\M_{0,n}^\delta$ as the union of those strata in $\MM_{0,n}$ whose corresponding dual graph arises from a collection of chords in an $n$-gon. 
	\end{enumerate}

\begin{figure}
	\centering
	\begin{minipage}{.45\linewidth}
		\centering
		
		\begin{tikzpicture}[font=\scriptsize, baseline={([yshift=-.5ex]current bounding box.center)}]
		\node[regular polygon, shape border rotate=45, regular polygon sides=8, minimum size=2cm, draw] at (5*2,0) (A) {};
		\foreach \i in {1,...,8} {
			\coordinate (ci) at (A.corner \i);
			\coordinate (si) at (A.side \i);
			\node at (ci) {};
			\path (A.center) -- (si) node[pos=1.25] {}; 
		}
		\path
		(A.corner 1) edge node[auto] {} (A.corner 3)
		(A.corner 4) edge node[auto] {} (A.corner 8)
		(A.corner 4) edge node[auto] {} (A.corner 6);
%
%
		
		\end{tikzpicture}
		\captionof{figure}{A configuration of three pairwise non-crossing chords in an octagon.}
		\label{chords1}
	\end{minipage}
	\hspace{.05\linewidth}
	\begin{minipage}{.45\linewidth}
		\centering
		
		\begin{tikzpicture}[font=\scriptsize, baseline={([yshift=-.5ex]current bounding box.center)}]
		\node[regular polygon, shape border rotate=45, regular polygon sides=8, dotted, minimum size=2cm, draw] at (5*2,0) (A) {};
		\foreach \i in {1,...,8} {
			\coordinate (ci) at (A.corner \i);
			\coordinate (si) at (A.side \i);
			\node at (ci) {};
			\path (A.center) -- (si) node[pos=1.25] {}; 
		}
		\path[dotted]
		(A.corner 1) edge node[auto] {} (A.corner 3)
		(A.corner 4) edge node[auto] {} (A.corner 8)
		(A.corner 4) edge node[auto] {} (A.corner 6);

		\coordinate (C0) at ($0.375*(A.corner 1)+0.25*(A.corner 2) + 0.375*(A.corner 3)$) {};
		\coordinate (C1) at ($0.25*(A.corner 3)+0.25*(A.corner 1) + 0.25*(A.corner 4) + 0.25*(A.corner 8)$) {};
		\coordinate (C2) at ($0.25*(A.corner 6)+0.25*(A.corner 7) + 0.25*(A.corner 4) + 0.25*(A.corner 8)$) {};
		\coordinate (C3) at ($0.31*(A.corner 4)+0.33*(A.corner 5) + 0.35*(A.corner 6)$) {};

		\path[thick]
		(C0) edge node[auto] {} (A.side 1)
		(C0) edge node[auto] {} (A.side 2)
		(C1) edge node[auto] {} (A.side 3)
		(C1) edge node[auto] {} (A.side 8)
		(C2) edge node[auto] {} (A.side 7)
		(C2) edge node[auto] {} (A.side 6)
		(C3) edge node[auto] {} (A.side 5)
		(C3) edge node[auto] {} (A.side 4)
		(C3) edge node[auto] {} (C2)
		(C1) edge node[auto] {} (C2)
		(C1) edge node[auto] {} (C0);
		
		\end{tikzpicture}
				
		\captionof{figure}{The Poincar\'e dual graph to the collection of chords.}
		\label{chords2}
	\end{minipage}
\end{figure}

Suppose that a given dual graph $\Gamma$ is compatible with the given dihedral structure, so that there exists an embedding of $\Gamma$ in $\R^2$ inducing the dihedral structure on its legs. Then this embedding is unique up to isotopy and orientation reversing. If we upgrade our dihedral structure to a cyclic ordering of $\{1,\ldots,n\}$, then there is a unique embedding of $\Gamma$ in $\R^2$ up to isotopy, inducing the given cyclic ordering. The data of such an embedding is the same as the structure of a \emph{ribbon graph} on $\Gamma$. Thus the strata in the stratification of $\M_{0,n}^\delta$ by topological type correspond bijectively to trees of exactly the same form as those in the stratification of $\MM_{0,n}$, but which are additionally equipped with a planar structure. Accordingly, the collection $\{\M_{0,n}^\delta\}_{n \geq 3}$ is a planar operad. 

In analogy to how we defined the operad \(\Hycom \), we define $\Prim$ to be the planar operad with $\Prim_n = H^{\bullet-2}(\M_{0,n}^\delta) \otimes \Q(-1)$ and composition given by Gysin maps.  Let $\chainPrim$ denote the planar operad $\{\CC_{2n-4-\bullet}(\M_{0,n}^\delta,D^n)\}_{n \geq 3}$, so that $\Prim = H(\chainPrim)$.

\begin{thm}\label{planar}\label{planarduality} We have an isomorphism $\cobarp \chaincoGrav \cong \chainPrim$ of planar operads of dg nuclear Fr\'echet spaces.
\end{thm}

\begin{proof} Repeat the proof of Theorem \ref{gk}, noting that $\M_{0,n}^\delta \setminus M_{0,n}$ is again a strict normal crossing divisor, but now the intersections of its components correspond to \emph{planar} stable trees with $n$ legs.\end{proof}

\subsection{Purity and formality of operads}

There is some history of results showing that if an algebraic variety $X$ has pure cohomology, then the topological space $X(\C)$ is \emph{formal}; that is,  $H^\bullet(X)$ is quasi-isomorphic to $A_{PL}^\bullet(X)$ (or any other cdga model for $X$) as a differential graded $\Q$-algebra. The following heuristic argument for why one might expect such a result is taken from the introduction to \cite{dgms}. Suppose that each cohomology group $H^k(X)$ is pure of weight $k$, e.g. if $X$ is a smooth projective variety. By \cite{kadeishvili}, $H^\bullet(X)$ is equivalent to $A_{PL}^\bullet(X)$ as an $A_\infty$-algebra, for some collection of $A_\infty$-operations $\{\mu_n\}_{n \geq 2}$ on $H^\bullet(X)$ with $\mu_2$ the usual cup product. But the operation $\mu_n$ has degree $2-n$, so if we believe that these operations can be made compatible with the weight filtrations, then all $\mu_n$ for $n \geq 3$ should vanish. Then $X$ must be formal. This heuristic was motivation for the following theorem:

\begin{thm}[Deligne--Griffiths--Morgan--Sullivan]\label{dgmsthm}
	Let $X$ be a compact K\"ahler manifold, e.g.\ a smooth projective variety. Then there is a canonical zig-zag of quasi-isomorphisms
	$$ H^\bullet(X,\R) \leftarrow \E^\bullet_{\mathrm{cl}}(X)\rightarrow \E^\bullet(X) $$
	where $\E^\bullet(X)$ denotes the $\mathcal C^\infty$ de Rham complex and $\E^\bullet_{\mathrm{cl}}(X)$ its subcomplex of $d^c$-closed forms. In particular, $X$ is formal. 
\end{thm}

If $H^k(X)$ is instead pure of weight $2k$ for all $k$, then the same heuristic is valid.  In this case a formality result can be obtained as follows:

\begin{thm}[Deligne] \label{2kformality}
	Let $X$ be a smooth algebraic variety, and ${\overline X}$ a smooth compactification of $X$ such that $D=\overline X \setminus X$  is a strict normal crossing divisor. Let $\Omega^\bullet_{\overline X}\langle D\rangle$ be the global sections of the logarithmic de Rham complex, i.e.\ the complex of meromorphic differential forms on ${\overline X}$ which are holomorphic on $X$ and have at most logarithmic singularities along $D$. Then  $\Omega^\bullet_{\overline X}\langle D\rangle$ has vanishing differential, the natural map
	$$  \Omega^\bullet_{\overline X}\langle D\rangle \to H^\bullet(X,\C)$$
	is an injection, and the image in degree $k$ equals $F^k H^k(X,\C)$, where $F$ denotes the Hodge filtration.  
\end{thm}

\begin{proof}This is a particular consequence of \cite[Corollaire 3.2.13(ii)]{hodge2}.
\end{proof}

\begin{cor}
	If $X$ is a smooth algebraic variety for which $H^k(X)$ is purely of weight $2k$ for all $k$, then $X$  is formal. 
\end{cor}

\begin{proof}
	The assumptions say that $H^k(X,\C)$ is of type $(k,k)$, so $ H^k(X,\C) = F^k H^k(X,\C)\cong \Omega^\bullet_{\overline X}\langle D\rangle$. Then the inclusion  of $\Omega^\bullet_{\overline X}\langle D\rangle$ as a subalgebra of the $\mathcal C^\infty$ de Rham complex of $X$ is a quasi-isomorphism. 
\end{proof}

An alternative proof of formality in the smooth projective case was given by Deligne  \cite[Section (5.3)]{weil2}, using $\ell$-adic cohomology. This latter proof is significantly closer to the heuristics outlined in the beginning of this subsection. 

Operads can be thought of as generalizations of associative algebras. The appropriate operadic generalization of an $A_\infty$-algebra is an \emph{operad up to homotopy} \cite{vanderlaan}. By \cite{granaker}, operads up to homotopy satisfy an analogue of Kadei{\v{s}}vili's theorem: if $P$ is a dg operad, then $H(P)$ is equipped with a collection of operations $\{\mu_n\}$ with $\mu_2$ the usual operadic composition, such that $P$ and $H(P)$ are equivalent as operads up to homotopy. Again $\mu_n$ has degree $2-n$. Thus it seems plausible that any dg operad $P$ of ``algebro--geometric origin'' whose cohomology can be equipped with a natural mixed Hodge structure (or structure of $\ell$-adic Galois representation) for which $H^k$ is pure of weight $k$ (or $2k$) should have $P$ and $H(P)$ quasi-isomorphic; that is, $P$ should be formal. In particular, we expect both $\chainHycom$ and $\chaincoGrav$ to be formal, by Proposition \ref{purityprop}. This is indeed the case:

\begin{thm}\label{formality}
	Both $\chainHycom$ and $\chaincoGrav$ are formal (co)operads.
\end{thm}

\begin{proof}
Note that $\chaincoGrav(n) = \E^{\bullet-1}(\MM_{0,n},D^n)$ contains $\Omega^{\bullet-1}_{\MM_{0,n}}\langle D^n\rangle$ as a subcomplex (with notation as in Subsection \ref{formscurrents} and Proposition \ref{2kformality}), and that the latter is isomorphic to $H^{\bullet-1}(M_{0,n},\C)$.  Since the Poincar\'e residue of a meromorphic form is meromorphic, these subcomplexes are preserved by the cooperadic cocomposition. This proves formality. 

Formality of $\chainHycom$ is a special case of  the results of \cite{formaloperads}. In brief, the complexes $\E^\bullet_{\mathrm{cl}}(X)$ and $\E^\bullet(X)$ of Theorem \ref{dgmsthm} satisfy a K\"unneth theorem and are functorial for holomorphic maps, which implies that the zig-zag of Theorem \ref{dgmsthm} (or rather its dual, in terms of the complex of currents) defines a quasi-isomorphism between the operad $\chainHycom$ with its cohomology $\Hycom$.  
\end{proof}

In particular, we observe that Theorem \ref{formality} and Theorem \ref{gk} combine to give a second proof of Theorem \ref{koszul1}, that $\coGrav$ and $\Hycom $ are Koszul dual to each other.

Since formality of the cooperad $\chaincoGrav$ relied so strongly on the fact that $H^k(M_{0,n})$ is pure of weight $2k$, and this fails in higher genus, it would seem likely that the \emph{modular} (higher genus) version of the cooperad $\chaincoGrav$ fails to be formal. This is indeed the case. In the following proof we use the \emph{Feynman transform}, which is just the version of the cobar construction that acts on modular operads. The Feynman transform interchanges $\mathfrak K^{-1}$-modular cooperads and modular operads, just like the bar--cobar transforms interchange anticyclic cooperads and cyclic operads. 

\begin{prop}
	The $\mathfrak K^{-1}$-modular cooperad $\mathfrak{coGrav}$ is not formal.
\end{prop}

\begin{proof}
	By \cite[Theorem 6.11]{getzlerkapranov}, the Feynman transform of $\mathfrak{coGrav}$ is isomorphic to the modular operad $\chainHycom$. In particular, this Feynman transform computes the cohomology of $\MM_{g,n}$. If $\mathfrak{coGrav}$ were formal, the same would be true for the Feynman transform of its cohomology $\mathsf{coGrav}$. But by reasoning as in the proof of Theorem \ref{koszul1}, we may identify the Feynman transform of $\mathsf{coGrav}$ with the $E_1$ page of the spectral sequence associated with the filtration of $\MM_{g,n}$ by topological type, which \emph{also} computes the cohomology of $\MM_{g,n}$. Thus if $\mathfrak{coGrav}$ were formal, the latter spectral sequence would have to degenerate at $E_2$, just by considerations of Betti numbers. This spectral sequence does degenerate at $E_2$ in genus zero, but not in general: as explained in \cite[Section 1]{genusone}, Getzler's relation on $\MM_{1,4}$ gives rise to a nonzero $E_2$-differential. 
\end{proof}

This answers questions raised in \cite[p. 3]{dsv} and in the end of the introduction of \cite{wardmc}.

\subsection{Equivalence of freedom and purity}\label{sec:equivalence}

We are now in a position to state the main result of this section. Before stating it, we remark that all of the equivalent statements in the following theorem are indeed true, and we have two independent proofs: Theorem \ref{thm:nscofree} in Section \ref{sec:freedom} shows that statement (1) below is satisfied, and Theorem \ref{thm:purity} in Section \ref{sec:purity} shows that condition (4) is true.  

\begin{thm}\label{equivalence}
	The following are equivalent:\begin{enumerate}
		\item The nonsymmetric cooperad $\coGrav$ is cofree.
		\item $\coGrav$ is cofree as an antiplanar cooperad, cogenerated by the collection $H^{\bullet-1}(M_{0,n}^\delta)$.
		\item The operad $\chainPrim$ is formal, and all composition maps on its cohomology operad $\Prim$ vanish.
		\item $H^i(M_{0,n}^\delta)$ has a pure Hodge structure of weight $2i$.
		\item $H^i(M_{0,n}^\delta) \to H^i(M_{0,n})$ is an injection.
	\end{enumerate}
\end{thm}

\begin{proof}
	\emph{(5) $\implies$ (4)} is clear, since we have already noted that $H^i(M_{0,n})$ has a pure Hodge structure of weight $2i$ (Proposition \ref{purityprop}).
	
	\emph{(4) $\implies$ (3).} There are quasi-isomorphisms $\chainPrim \cong \cobarp \chaincoGrav \simeq \cobarp \coGrav$ by Theorem \ref{planar} and formality of $\chaincoGrav$. Purity of $\Prim$ implies that the cohomology of $\cobarp \coGrav$ is concentrated in the summand corresponding to trees with a single vertex; that is, the edge map $\Prim \to \cobarp \coGrav$ is a quasi-isomorphism. 
	
	Moreover, the assumption implies that all the composition maps in $\Prim$ go between cohomology groups of different weights, as $\Prim^k_n = H^{k-2}(\M_{0,n}^\delta) \otimes \Q(-1)$ is of weight $2k-2$. The composition maps must therefore vanish. 
	
	\emph{(3) $\implies$ (2)} We have a chain of quasi-isomorphisms $\coGrav \simeq \chaincoGrav \simeq \barp \chainPrim \simeq \barp \Prim \cong \freeantipl(\Sigma\Prim)$, using (respectively) formality of $\chaincoGrav$, Theorem \ref{planar} and bar-cobar duality, the assumption that $\chainPrim$ is formal, and the assumption that the composition maps in $\Prim$ are zero.  
	
	\emph{(2) $\implies$ (1)} is trivial.
	
	\emph{(1) $\implies$ (5)}. By assumption we have $\coGrav = \mathrm{T}^\mathrm{ns}(M) = \mathrm{B}^{\mathrm{ns}}(\Sigma^{-1}M)$ for some collection $M$, where $\mathrm{T}^\mathrm{ns}$ denotes the cofree conilpotent nonsymmetric cooperad functor. By the nonsymmetric version of Theorem \ref{planar}, bar-cobar duality and formality of $\chaincoGrav$, it follows that $\Sigma^{-1}M \simeq \mathrm{\Omega}^{\mathrm{ns}} \coGrav \simeq \chainPrim$. Taking cohomology we see that $\Sigma^{-1}M \cong \Prim$. But the inclusion of $\Sigma\Prim$ into $\mathrm{B}^{\mathrm{ns}} \Prim$ as the summand corresponding to trees with a single vertex is an injection, and then so must $H^i(M_{0,n}^\delta) \to H^i(M_{0,n})$ be. (The identification of the two maps $H^\bullet(M_{0,n}^\delta) \to H^\bullet(M_{0,n})$ and $\Sigma\Prim \to \mathrm{B}^{\mathrm{ns}} \Prim$ follows from our identification of said bar construction with the Leray spectral sequence for $\M_{0,n} \to \M_{0,n}^\delta$.) \end{proof}

\begin{rem}
	The fact that the composition maps in the operad $\Prim$ are all zero can be given an easy proof independent of the rest of the results in this paper. We need to show that the Gysin maps
	$$ \pi_\ast \colon H^k(\M_{0,n+1}^\delta \times M_{0,m+1}^\delta) \to H^{k+2}(M_{0,n+m}^\delta)$$
	are all zero. When $k=0$, this is the same as saying that the cohomology class of the boundary divisor $D = \M_{0,n+1}^\delta\times M_{0,m+1}^\delta \subset \M_{0,n+m}^\delta$ is zero. If this boundary divisor corresponds to a chord in an $(n+m)$-gon, consider the WDVV relation on $\MM_{0,n+m}$ corresponding to the $4$ marked points that are ``adjacent'' to this chord. This is a linear relation between boundary divisors on $\MM_{0,n+m}$, all of which except $D$ are outside of $\M_{0,n+m}^\delta$. Thus $[D]=0$ in $H^2(\M_{0,n+m}^\delta)$. 
	
	Secondly, we observe that the pullback map $\pi^\ast \colon H^\bullet(M_{0,n+m}^\delta) \to H^\bullet(\M_{0,n+1}^\delta\times M_{0,m+1}^\delta)$ is surjective. Indeed, the inclusion of this divisor is a retract; a left inverse is given by a product of two forgetful maps. 
	
	Finally, to prove this vanishing also in higher degrees, we use the projection formula. Take $\alpha \in H^\bullet(\M_{0,n+1}^\delta\times M_{0,m+1}^\delta)$ and let $\alpha = \pi^\ast(\beta)$ for some $\beta$. Then
	$$\pi_\ast(\alpha) = \pi_\ast(1 \cdot \pi^\ast(\beta)) = \pi_\ast(1) \cdot \beta = 0,$$
	since we have already verified that $\pi_\ast(1) = 0$. 	
\end{rem}

\section{Freedom}\label{sec:freedom}

\subsection{Arrangements of hyperplanes}
\newcommand{\Jarc}{J_n^{\mathrm{arc}}}
\newcommand{\Jchord}{J_n^{\mathrm{chord}}}
Let $\{H_c\}_{c \in C}$ be a finite set of affine hyperplanes in $\C^n$. Let us write $U = \C^n \setminus \bigcup_{c \in C} H_c$ for the complement of the hyperplane arrangement. We denote by $f_c$ the linear form corresponding to the hyperplane $H_c$. By a theorem of Brieskorn \cite{brieskorn}, the cohomology ring $H^\bullet(U,\Z)$ is isomorphic to the subalgebra of the de Rham complex of $U$ generated by the $1$-forms $\frac{1}{2i\pi} d \log (f_c)$. Let $\wedge^\bullet C$ denote the exterior algebra generated by the set $C$. We write $e_c$ for the generator corresponding to $c \in C$, and we put
$$ e_S = \bigwedge_{c \in S} e_c$$
for $S \subseteq C$. (Unless $C$ is ordered, $e_S$ is only well defined up to a sign.) Thus we have the canonical surjection $\wedge^\bullet C \to H^\bullet(U,\Z)$, taking $e_c$ to the differential form $\frac{1}{2i\pi} d \log (f_c)$. 
Let $J$ denote the kernel of this surjection. One can define combinatorially a set of generators for $J$. For this we shall need some terminology.

\begin{defn}
	A subset $S \subseteq C$ is \emph{dependent} if the intersection of the corresponding hyperplanes is not transverse. An inclusion-minimal dependent subset is called a \emph{circuit.}
\end{defn}
\begin{defn}
	Suppose given a total order $\prec$ on $C$. A \emph{broken circuit} is a subset of $C$ of the form $A \setminus \min(A)$, where $A$ is a circuit. 
\end{defn}
\begin{defn}
	A subset $S \subseteq C$ is called an \emph{nbc-set} if no subset of $S$ forms a broken circuit and $\bigcap_{c \in S} H_c \neq \varnothing$. 
\end{defn}

 Let $\partial \colon \wedge^kC \to \wedge^{k-1}C$ be the Koszul differential, i.e. the unique derivation with $\partial(e_c)=1$ for all $c \in C$. 

The following two theorems are fundamental in the theory of hyperplane arrangements. For a textbook treatment we recommend \cite[Chapter 3]{orlikterao}. The original references (in the case of a central arrangement) are \cite{orliksolomon,gelfandzelevinskiihypergeometric,jambuterao}.

\begin{thm}[Orlik--Solomon] \label{orliksolomon}The ideal $J$ is generated by the elements $\partial e_A$, where $A$ ranges over the circuits in $C$, and $e_S$, where $S$ ranges over the subsets of $C$ for which $\bigcap_{c \in S} H_c = \varnothing$. 
\end{thm}

\begin{thm}[Gel'fand--Zelevinski\u\i, Jambu--Terao] \label{nbc} The elements $e_S$, where $S$ ranges over the nbc-sets in $C$, give a basis for $H^\bullet(U,\Z)$. 
\end{thm}

\begin{rem}\label{osrelations}The easier half of Theorem \ref{nbc} is that the elements $e_S$, where $S$ ranges over the nbc-sets in $C$, \emph{span} $H^\bullet(U,\Z)$. Indeed, if $A$ is a circuit, then the equation $\partial e_A = 0$ is a linear relation in which exactly one of the terms corresponds to a broken circuit. Thus the Orlik--Solomon relations of Theorem \ref{orliksolomon} can be successively used to eliminate broken circuits. In doing so, we may introduce new broken circuits; nevertheless, this procedure must terminate, since we are always replacing monomials with ones that are strictly smaller with respect to (say) the lexicographic order on the set of monomials with respect to $\prec$.\end{rem} 

Suppose that $S$ is a polynomial or exterior algebra equipped with a term order $\prec$ on its monomials. For $f \in S$, we write $\mathrm{in}(f)$ for the \emph{initial term} of $f$, i.e.\ the monomial which is largest with respect to $\prec$. If $I \subset S$ is an ideal, then we write $\mathrm{in}(I)$ for its \emph{initial ideal}, the ideal generated by $\mathrm{in}(f)$ for $f \in I$. Recall that a \emph{Gr\"obner basis} for $I$ is a set of generators $f_1,\ldots,f_k$ for $I$ such that $\mathrm{in}(f_1),\ldots,\mathrm{in}(f_k)$ generate $\mathrm{in}(I)$. A \emph{standard monomial} is a monomial in $S$ which is not in $\mathrm{in}(I)$. The set of all standard monomials forms a basis for the quotient $S/I$. 

Theorem \ref{nbc} may be reformulated in these terms by saying that the generators for the ideal $J$ given by Theorem \ref{orliksolomon} are in fact a \emph{Gr\"obner basis} for $J$ over the integers, for any choice of total order $\prec$. (It will not in general be a reduced Gr\"obner basis.) Indeed, for any circuit $A$, 
$$ \mathrm{in}(\partial e_A) = \pm e_{A \setminus \min(A)}, $$
so a broken circuit is nothing but the leading term of one of the Orlik--Solomon relations. Thus the basis for $H^\bullet(U,\Z)$ of nbc-sets is the basis of standard monomials with respect to this Gr\"obner basis. 

\subsection{Moduli space $M_{0,n}$ and arc diagrams} \label{sectionarc} Consider the moduli space $M_{0,n}$, parametrizing distinct ordered points $z_1,\ldots,z_{n}$ on $\P^1$ modulo the action of $\mathrm{PGL}(2)$. Using the gauge freedom to fix the marked points \(z_1,\,z_{n-1},\,z_{n} = 0,\,1,\,\infty\) leaves a configuration of distinct points $z_2,\ldots,z_{n-2}$ in $\C \setminus \{0,1\}$.  This gives an identification of $M_{0,n}$ with the complement of a hyperplane arrangement. We will find it convenient to write the corresponding affine space as
$$ \{(z_1,z_2,\ldots,z_{n-1}) \in \C^{n-1} : z_1=0, z_{n-1}=1\}$$
in which case the hyperplanes can be written in a uniform way as $\{z_i-z_j=0\}$, for $1 \leq i, j \leq n-1$. Let us write $C_n$ for the set of unordered pairs $\{i,j\}$ with $1 \leq i , j \leq n-1$, $i \neq j$, and $\{i,j\}\neq \{1,n-1\}$. Then our set of hyperplanes can be written as $\{H_{i,j}\}_{\{i,j\} \in C_n}$. 

We define the \emph{length} of the pair $\{i,j\}$ to be $\vert i - j \vert$. Let $\prec$ be an arbitrary total order on $C_n$ refining the partial order by \emph{reverse} length; that is, if $\vert i - j \vert > \vert k - l \vert$, then $\{i,j\} \prec \{k,l\}$.

Let $e_{ij}$ be the generator for $\wedge^\bullet C_n$ corresponding to $\{i,j\}$. We draw basis elements for $\wedge^\bullet C_n$ graphically by marking the points \(1,\dots,n-1\in \R\) as vertices on the boundary of an upper half plane \(\R\times\R_{\geq 0}\), and for each generator \(e_{ij}\) in the monomial we draw an arc in the upper half-plane with endpoints \(i\) and \(j\). We call such a basis element an \emph{arc diagram}. We write $\omega_{ij}$ for the differential form $\frac{1}{2i\pi} d\log(z_i-z_j)$, so that the natural map $\wedge^\bullet C_n \to H^\bullet(M_{0,n})$ takes $e_{ij}$ to the class represented by the differential form $\omega_{ij}$. Let $\Jarc$ denote the kernel of this map.

As an example of our notation, the ``Arnol'd relation'' \(\omega_{ij}\omega_{jk} + \omega_{ik}\omega_{ij} - \omega_{ik}\omega_{jk} = 0\) (which is exactly the Orlik--Solomon relation corresponding to the circuit $\{\{i,j\},\{i,k\},\{j,k\}\} \subset C_n$) can be interpreted as saying that the sum
\[
\begin{tikzpicture}[baseline=0ex,shorten >=0pt,auto,node distance=2cm]
\node[ve] (z1) {};
\node[ve] (z2) [right of=z1] {};
\node[ve] (z3) [right of=z2] {};  
\path[every node/.style={font=\sffamily\small}]
(z1) edge[out=75,in=105] (z2)
(z2) edge[out=75,in=105] (z3);   
\end{tikzpicture} 
+
\begin{tikzpicture}[baseline=0ex,shorten >=0pt,auto,node distance=2cm]
\node[ve] (z1) {};
\node[ve] (z2) [right of=z1] {};
\node[ve] (z3) [right of=z2] {};  
\path[every node/.style={font=\sffamily\small}]
(z1) edge[out=75,in=105] (z3)
(z1) edge[out=75,in=105] (z2);
\end{tikzpicture}
-
\begin{tikzpicture}[baseline=0ex,shorten >=0pt,auto,node distance=2cm]
\node[ve] (z1) {};
\node[ve] (z2) [right of=z1] {};
\node[ve] (z3) [right of=z2] {};  
\path[every node/.style={font=\sffamily\small}]
(z1) edge[out=75,in=105] (z3)
(z2) edge[out=75,in=105] (z3);   
\end{tikzpicture}
\] 
goes to zero under $\wedge^\bullet C_n \to H^\bullet (M_{0,n})$; the three vertices in each term are assumed to be labeled $i$, $j$ and $k$, respectively. 

We observe that a circuit in $C_n$ is exactly an arc diagram for which the corresponding graph is a closed cycle, such as the following three arc diagrams:
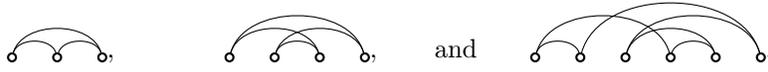
\[
\begin{tikzpicture}[baseline=0ex,shorten >=0pt,auto,node distance=2cm]
\node[ve] (z1) {};
\node[ve] (z2) [right of=z1] {};
\node[ve] (z3) [right of=z2] {};  
\path[every node/.style={font=\sffamily\small}]
(z1) edge[out=75,in=105] (z3)
(z2) edge[out=75,in=105] (z3)
(z1) edge[out=75,in=105] (z2);    
\end{tikzpicture}, \qquad \qquad
\begin{tikzpicture}[baseline=0ex,shorten >=0pt,auto,node distance=2cm]
\node[ve] (z1) {};
\node[ve] (z2) [right of=z1] {};
\node[ve] (z3) [right of=z2] {};  
\node[ve] (z4) [right of=z3] {};  
\path[every node/.style={font=\sffamily\small}]
(z1) edge[out=75,in=105] (z4)
(z2) edge[out=75,in=105] (z4)
(z2) edge[out=75,in=105] (z3)
(z1) edge[out=75,in=105] (z3);    
\end{tikzpicture},\qquad \text{and} \qquad
\begin{tikzpicture}[baseline=0ex,shorten >=0pt,auto,node distance=2cm]
\node[ve] (z1) {};
\node[ve] (z2) [right of=z1] {};
\node[ve] (z3) [right of=z2] {};  
\node[ve] (z4) [right of=z3] {};
\node[ve] (z5) [right of=z4] {};  
\node[ve] (z6) [right of=z5] {};    
\path[every node/.style={font=\sffamily\small}]
(z1) edge[out=75,in=105] (z4)
(z3) edge[out=75,in=105] (z6)
(z2) edge[out=75,in=105] (z6)
(z1) edge[out=75,in=105] (z2)
(z3) edge[out=75,in=105] (z5)
(z4) edge[out=75,in=105] (z5);    
\end{tikzpicture} .
\]
A broken circuit is an arc diagram from which the $\prec$-smallest arc has been removed, and thus the corresponding graph is a \emph{path}. The broken circuits corresponding to the above three circuits would be
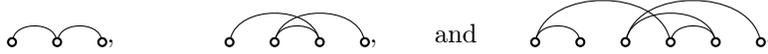
\[
\begin{tikzpicture}[baseline=0ex,shorten >=0pt,auto,node distance=2cm]
\node[ve] (z1) {};
\node[ve] (z2) [right of=z1] {};
\node[ve] (z3) [right of=z2] {};  
\path[every node/.style={font=\sffamily\small}]
(z2) edge[out=75,in=105] (z3)
(z1) edge[out=75,in=105] (z2);    
\end{tikzpicture}, \qquad \qquad
\begin{tikzpicture}[baseline=0ex,shorten >=0pt,auto,node distance=2cm]
\node[ve] (z1) {};
\node[ve] (z2) [right of=z1] {};
\node[ve] (z3) [right of=z2] {};  
\node[ve] (z4) [right of=z3] {};  
\path[every node/.style={font=\sffamily\small}]
(z2) edge[out=75,in=105] (z4)
(z2) edge[out=75,in=105] (z3)
(z1) edge[out=75,in=105] (z3);    
\end{tikzpicture},\qquad \text{and} \qquad
\begin{tikzpicture}[baseline=0ex,shorten >=0pt,auto,node distance=2cm]
\node[ve] (z1) {};
\node[ve] (z2) [right of=z1] {};
\node[ve] (z3) [right of=z2] {};  
\node[ve] (z4) [right of=z3] {};
\node[ve] (z5) [right of=z4] {};  
\node[ve] (z6) [right of=z5] {};    
\path[every node/.style={font=\sffamily\small}]
(z1) edge[out=75,in=105] (z4)
(z3) edge[out=75,in=105] (z6)
(z1) edge[out=75,in=105] (z2)
(z3) edge[out=75,in=105] (z5)
(z4) edge[out=75,in=105] (z5);    
\end{tikzpicture} .
\]
In these three examples there was a unique arc of maximal length in the circuit, which is then necessarily the smallest arc under $\prec$; recall that we allowed $\prec$ to be an arbitrary linear extension of the partial order by reverse arc-length. 

\begin{prop} \label{arcdiagram} The ``no broken circuit''-basis of $H^\bullet(M_{0,n})$ provided by Theorem \ref{nbc} consist exactly of those arc diagrams  which do not contain vertices $i<j<k$ such that the monomial has a factor $e_{ij}e_{jk}$, and which do not contain vertices $i<j<k<l$ such that $e_{ik}e_{jk}e_{jl}$ is a factor.\end{prop}

\begin{proof}It is clear that if an arc diagram has a factor $e_{ij}e_{jk}$ or $e_{ik}e_{jk}e_{jl}$, then it contains a broken circuit; specifically, these are exactly the first two broken circuits given as examples before the lemma.
	
	Conversely, we need to show that any broken circuit, and any arc diagram for which the corresponding set of hyperplanes has empty intersection, contains one of these two as a factor. Take first an arc diagram corresponding to a broken circuit, and suppose it does not contain any factor of the form 
	\[
	\begin{tikzpicture}[baseline=0ex,shorten >=0pt,auto,node distance=2cm]
	\node[ve] (z1) {};
	\node[ve] (z2) [right of=z1] {};
	\node[ve] (z3) [right of=z2] {};  
	\path[every node/.style={font=\sffamily\small}]
	(z2) edge[out=75,in=105] (z3)
	(z1) edge[out=75,in=105] (z2);    
	\end{tikzpicture}.\]
	This means that the path given by the arc diagram switches directions from left to right at each step. In particular, it is not possible for two successive arcs to have the same length.  Now suppose also that the broken circuit does not  have a factor
	\[\begin{tikzpicture}[baseline=0ex,shorten >=0pt,auto,node distance=2cm]
	\node[ve] (z1) {};
	\node[ve] (z2) [right of=z1] {};
	\node[ve] (z3) [right of=z2] {};  
	\node[ve] (z4) [right of=z3] {};  
	\path[every node/.style={font=\sffamily\small}]
	(z2) edge[out=75,in=105] (z4)
	(z2) edge[out=75,in=105] (z3)
	(z1) edge[out=75,in=105] (z3);    
	\end{tikzpicture}.\]
	Equivalently, at no step of the path is there an arc which has shorter length than both its neighbors. It follows that the sequence of arc-lengths along the path is \emph{unimodal}: the arc-lengths are first strictly increasing, until they reach a maximum and become strictly decreasing. Thus the path would look e.g. as follows: 
	
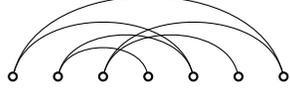
\[\begin{tikzpicture}[baseline=0ex,shorten >=0pt,auto,node distance=2cm]
	\node[ve] (z1) {};
	\node[ve] (z2) [right of=z1] {};
	\node[ve] (z3) [right of=z2] {};  
	\node[ve] (z4) [right of=z3] {}; 
	
	\node[ve] (z5) [right of=z4] {}; 
	\node[ve] (z6) [right of=z5] {}; 
	\node[ve] (z7) [right of=z6] {};  
	\path[every node/.style={font=\sffamily\small}]
	(z1) edge[out=75,in=105] (z7)
	(z2) edge[out=75,in=105] (z4)
	(z3) edge[out=75,in=105] (z6)
	(z3) edge[out=75,in=105] (z7)
	(z2) edge[out=75,in=105] (z5)
	(z1) edge[out=75,in=105] (z5);    
	\end{tikzpicture}.\]
	But it is now clear that this could not have been a broken circuit: the missing arc that is needed to close the path to a cycle is completely contained within the arc of maximal length in the path.

	Suppose instead that we have an arc diagram for which the intersection of the corresponding hyperplanes is empty. This can only happen if the vertex ``1'' and the vertex ``$(n-1)$'' are in the same connected component of the graph given by the arc diagram. Choose a path in the arc diagram from the vertex $1$ to the vertex $(n-1)$. Since the arcs in the diagram correspond to elements of $C_n$, and $\{1,n-1\} \not\in C_n$, the path must have length at least $2$. If this path does not contain a broken circuit, then the same argument as before shows that the edge-lengths must be unimodal, which is clearly impossible. \end{proof}

\begin{defn}
	We define a \emph{gravity arc diagram} to be an arc diagram satisfying the conditions of the above proposition; that is, the set of gravity arc diagrams equals the set of basis elements in the ``no broken circuit''-basis for $H^\bullet(M_{0,n})$. 
\end{defn}

\begin{rem}\label{arcdiagramremark}
	Proposition \ref{arcdiagram} may be reformulated as saying that the initial ideal of $\Jarc$ is generated by the monomials $e_{ij}e_{jk}$ for $i<j<k$, and $e_{ik}e_{jk}e_{jl}$ for $i<j<k<l$. 
\end{rem}

\subsection{Chords and cross-ratios}By a \emph{chord} in a convex polygon we mean a straight line connecting two non-consecutive vertices. Let us label the sides of the standard $n$-gon clockwise from $1$ to $n$. Each chord partitions the set $\{1,\ldots,n\}$ into two intervals (for the cyclic order). We shall label each chord by the two endpoints of the corresponding interval in $\{1,\ldots,n\}$ \emph{not} containing the element $n$. Below is the standard \(n\)-gon for \(n=5\), with the chord \(\{1,3\}\) drawn on it.
	
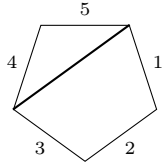
\[
	\begin{tikzpicture}[font=\scriptsize, baseline={([yshift=-.5ex]current bounding box.center)}]
	\node[regular polygon, shape border rotate=36, regular polygon sides=5, minimum size=2cm, draw] at (5*2,0) (A) {};
	\foreach \i [count=\j] in {4,...,1,5} {
		\coordinate (si) at (A.side \j);
		\path (A.center) -- (si) node[pos=1.25] {\i}; 
	}
	\path[thick] (A.corner 2) edge node[auto] {} (A.corner 5);
	\end{tikzpicture}
	\]
Under this labeling, the set of chords becomes identified with the set of pairs $\{i,j\}$ of elements in $\{1,\ldots,n-1\}$ except $\{1,n-1\}$, which is exactly the set $C_n$ that we introduced in the previous section. Recall that we defined the \emph{length} of an element of $C_n$, and a total order $\prec$ on $C_n$; we will use this to freely talk about the relation $\prec$ also on the set of chords.

Each chord is incident to four sides of the polygon, which correspond to four of the markings on $M_{0,n}$. Thus each chord defines a forgetful map $u_{ij} \colon M_{0,n} \to M_{0,4} \cong \A^1\setminus \{0,1\}$, or more classically, a cross-ratio. More precisely, we let $u_{ij}$ for $i<j$ be the cross-ratio $[\,i \,\,i-1\,\vert\, j \,\, j+1\,]$, where subtraction and addition is taken modulo $n$.

We caution the reader that our numbering of the chords does \emph{not} agree with the one used in \cite{brownmzv}: our $u_{ij}$ is his $u_{i-1\,j}$. 
	
	There is an evident notion of when two chords \emph{cross}. Say that two subsets \(A,B\subset C_{n}\) are \emph{completely crossing} if whenever a chord \(c\) crosses all chords in \(A\), then \(c\in B\), and vice versa, if a chord \(c\) crosses all chords in \(B\), then \(c\in A\). These notions can be used to give a presentation of the coordinate ring of $\M_{0,n}$ \cite[Section 2.2]{brownmzv}: one has
	\[
	M_{0,n} = \mathrm{Spec}\,\Q[u_c^{\pm 1}\mid c\in C_{n}]/\langle R \rangle,
	\]
	where \(R\) is the set of relations that
	\(
	\prod_{a\in A} u_a + \prod_{b\in B} u_b = 1
	\)
	for all pairs of completely crossing subsets \(A,B\subset C_n\). The codimension 1 strata are given by equations \(u_c=0\). For example, the stratum \(u_{\{1,3\}}=0\) (corresponding to the chord shown in the picture above) corresponds to the partition into intervals \(\{5,4\}\sqcup\{1,2,3\}\). Strata of higher codimension correspond to tilings and are explicitly given by setting \(u_c=0\) for all chords \(c\) in the tiling.

This leads moreover to a presentation of the cohomology ring. The cohomology ring $H^\bullet(M_{0,n},\Z)$ is isomorphic to the subalgebra of the de Rham complex generated by the differential forms
$$ \alpha_c = \frac{1}{2i\pi}d\log u_c, \qquad c \in C_n,$$
which gives us a \emph{second} natural surjective map $\wedge^\bullet C_n \to H^\bullet(M_{0,n},\Z)$. Let $\Jchord$ denote the kernel of this surjection. A set of generators for $\Jchord$, and hence a different presentation of the cohomology ring $H^\bullet(M_{0,n},\Z)$, was determined in \cite[Proposition 6.2]{brownmzv}: the ideal $\Jchord$ is generated by all expressions $$\biggl(\,\sum_{a\in A} \alpha_a\biggr)\! \biggl(\,\sum_{b\in B} \alpha_b\biggr),$$ where $A, B \subset C_n$ are completely crossing subsets.
\subsection{The basis of chord diagrams}

Let $R$ be a ring, and let $S$ be the exterior algebra over $R$ generated by elements $x_1,\ldots,x_n$. (Everything we say now is true also in the more classical case of a polynomial algebra.) Suppose we fix the ordering $x_1 \prec x_2 \prec \ldots \prec x_n$, and extend this to a term ordering on the monomials in the generators (e.g. by the lexicographic order). There is a natural action of $G = GL_n(R)$ on $S$: an element $g = (g_{ij}) \in G$ acts on the generators via the rule
$$ g \cdot x_j = \sum_{i=1}^n g_{ij} x_i.$$
Let $U \subset G$ be the unipotent subgroup of upper triangular matrices with ones along the diagonal. The action of $G$ on the initial ideals and Gr\"obner bases is nontrivial and interesting, and leads to the theory of generic initial ideals. However, all the action takes place in the coset space $U\backslash G$, by the following easy lemma:
\begin{lem}\label{grobnerlemma}
	For any $f \in S$ and $u \in U$, we have $\mathrm{in}(f) = \mathrm{in}(u \cdot f)$. In particular, $I$ and $u \cdot I$ have the same initial ideal for any ideal $I$, and if $f_1,\ldots,f_k$ is a Gr\"obner basis for $I$, then $u \cdot f_1,\ldots, u \cdot f_k$ is a Gr\"obner basis for $u \cdot I$.  
\end{lem}

\begin{proof}
	If $x^\alpha$ is a monomial in $S$, then \[ u \cdot x^\alpha = x^\alpha + (\text{terms strictly $\prec$-smaller than $x^\alpha$}). \]
	This implies the result.
\end{proof}

After this brief interlude, let us return to the problem at hand. In terms of the coordinates $z_1,\ldots,z_{n}$ on $M_{0,n}$ defined in Subsection \ref{sectionarc}, the differential forms $\alpha_{ij}$ may be written in the form
\[
\alpha_{ij} = \frac{1}{2i\pi}d\log \frac{(z_{i}-z_j)(z_{i-1}-z_{j+1})}{(z_{i-1}-z_{j})(z_i-z_{j+1})} = 
\omega_{ij}+\omega_{i-1\,j+1}-\omega_{i-1\,j}-\omega_{i\,j+1}.
\]
The reader may check that this is valid also if one of the four points equals $z_n=\infty$, in which case we must set $\omega_{in} = 0 =\omega_{1\,n-1}$. Thus there is a commutative diagram
	\[
	\begin{tikzcd}[column sep=1em]
	\wedge^\bullet C_n \arrow{rr}{u}\arrow{dr}{f} && \wedge^\bullet C_n \arrow{ld}[swap]{g} \\
	& H^\bullet(M_{0,n},\Z) &
	\end{tikzcd}
	\]
	where $f$ takes a generator $e_{ij}$ to the class $\alpha_{ij}$, $g$ takes $e_{ij}$ to $\omega_{ij}$, and
	$$ u(e_{ij}) = e_{ij} + e_{i-1\,j+1} - e_{i-1\, j} - e_{i\,j+1}.$$ Again we tacitly set $e_{1\,n-1} = e_{in} = 0$. But now note that in the expression $e_{ij} + e_{i-1\,j+1} - e_{i-1\, j} - e_{i\,j+1}$, the first term is strictly larger than the remaining three under $\prec$. This implies that we are in the situation of Lemma \ref{grobnerlemma}, and we deduce:
	
	\begin{thm}The initial ideal $\mathrm{in}(\Jchord)$ is generated by the monomials $e_{ij}e_{jk}$ for $i<j<k$, and $e_{ik}e_{jk}e_{jl}$ for $i<j<k<l$. 
	\end{thm}
\begin{proof}
	The map $u$ is an automorphism of the algebra $\wedge^\bullet C_n$, taking the ideal $\Jchord$ (the kernel of $f$) to the ideal $\Jarc$ (the kernel of $g$). But since $u$ is an automorphism of the form considered in Lemma \ref{grobnerlemma}, this implies that $\Jchord$ and $\Jarc$ have the same initial ideal. The initial ideal of $\Jarc$ was determined in Proposition \ref{arcdiagram}, see Remark \ref{arcdiagramremark}. 
\end{proof}

We define a \emph{chord diagram} on the \(n\)-gon to be a basis element in the exterior algebra $\wedge^\bullet C_n$, drawn as a collection of chords on a standard \(n\)-gon. Thus the sets of arc diagrams and chord diagrams are in canonical bijection with each other, and with the set of subsets of $C_n$. The difference between arc diagrams and chord diagrams is that they are drawn diagrammatically in different ways, and that when we think of an arc diagram as a cohomology class on $M_{0,n}$ we are applying the map $e_{ij} \mapsto \omega_{ij}$, whereas arc diagrams are interpreted as cohomology classes via $e_{ij} \mapsto \alpha_{ij}$.

\begin{defn}
	A \emph{gravity chord diagram} is a chord diagram which is not divisible by the monomials $e_{ij}e_{jk}$ for $i<j<k$, and $e_{ik}e_{jk}e_{jl}$ for $i<j<k<l$. 
\end{defn}

\begin{cor}
	The set of gravity chord diagrams 
	defines a basis for $H^\bullet(M_{0,n})$. 
\end{cor}

\subsection{Cooperad structure of chord diagrams}
Let us give a graphical interpretation of gravity chord diagrams. Recall that the side of the $n$-gon labeled by $n$ played a special role in the way we numbered the chords on the $n$-gon: each chord partitions the set $\{1,\ldots,n\}$ into two intervals for the cyclic ordering, and the chord is labeled by the two endpoints of the interval \emph{not} containing $n$. This corresponds to the fact that gravity chord diagrams will form a basis for $\coGrav$ as a nonsymmetric cooperad (as opposed to an antiplanar one), so the condition of being a gravity chord diagram should depend on the choice of a total ordering of the sides, or equivalently, the choice of a distinguished side of the polygon. Geometrically this corresponds to fixing a point at infinity. 

The condition for being a gravity chord diagram is then the following: for every pair of crossing chords in the chord diagram, consider the corresponding inscribed quadrilateral. The side of this quadrilateral that is opposite from the distinguished side of the polygon is not allowed to be a side of the polygon, nor is it allowed to be a chord in the diagram. The two forms of inadmissible chord diagrams are illustrated in Figures \ref{grav1} and \ref{grav2}; the ``inscribed quadrilaterals'' mentioned in the definition are depicted by dotted lines, and the distinguished side of the polygon is the one on the top.

\begin{figure}[!h]
	\centering
	\begin{minipage}{.45\linewidth}
		\centering
		\begin{tikzpicture}[font=\scriptsize, baseline={([yshift=-.5ex]current bounding box.center)}]
		\node[regular polygon, shape border rotate=45, regular polygon sides=8, minimum size=2cm, draw] at (5*2,0) (A) {};
		\foreach \i in {1,...,8} {
			\coordinate (ci) at (A.corner \i);
			\coordinate (si) at (A.side \i);
			\node at (ci) {};
			\path (A.center) -- (si) node[pos=1.25] {}; 
		}
		\path[densely dotted]
		(A.corner 2) edge node[auto] {} (A.corner 4)
		(A.corner 2) edge node[auto] {} (A.corner 7)
		(A.corner 5) edge node[auto] {} (A.corner 7);
		\path[thick] 
		(A.corner 2) edge node[auto] {} (A.corner 5)
		(A.corner 4) edge node[auto] {} (A.corner 7);
		\path[very thick]
		(A.corner 1) edge node[auto] {} (A.corner 8);
		\end{tikzpicture}
		
		\captionof{figure}{A chord diagram of the form $e_{ij}e_{jk}$ for $i<j<k$.}
		\label{grav1}
	\end{minipage}\hspace{0.05\linewidth}\begin{minipage}{.45\linewidth}
		\centering
		\begin{tikzpicture}[font=\scriptsize, baseline={([yshift=-.5ex]current bounding box.center)}]
		\node[regular polygon, shape border rotate=45, regular polygon sides=8, minimum size=2cm, draw] at (5*2,0) (A) {};
		\foreach \i in {1,...,8} {
			\coordinate (ci) at (A.corner \i);
			\coordinate (si) at (A.side \i);
			\node at (ci) {};
			\path (A.center) -- (si) node[pos=1.25] {}; 
		}
		\path[densely dotted]
		(A.corner 7) edge node[auto] {} (A.corner 5)
		(A.corner 1) edge node[auto] {} (A.corner 3)
		(A.corner 1) edge node[auto] {} (A.corner 7);
		\path[thick] 
		(A.corner 1) edge node[auto] {} (A.corner 5)
		(A.corner 3) edge node[auto] {} (A.corner 7)
		(A.corner 3) edge node[auto] {} (A.corner 5);
		\path[very thick]
		(A.corner 1) edge node[auto] {} (A.corner 8);
		\end{tikzpicture} 
		
		\captionof{figure}{A chord diagram of the form $e_{ik}e_{jk}e_{jl}$ for $i<j<k<l$.} 
		\label{grav2}
	\end{minipage}
\end{figure}

A diagram with $i$ chords in an $n$-gon and a diagram with $j$ chords in an $m$-gon can be grafted together to produce a  diagram with $(i+j+1)$ chords in an $(n+m-2)$-gon, by identifying two sides of the polygons with each other and including this side as a chord in the $(n+m-2)$-gon. As an example, grafting the two chord diagrams
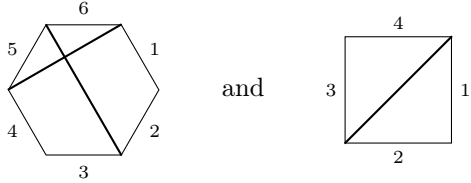
\[
\begin{tikzpicture}[font=\scriptsize, baseline={([yshift=-.5ex]current bounding box.center)}]
\node[regular polygon, shape border rotate=60, regular polygon sides=6, minimum size=2cm, draw] at (5*2,0) (A) {};
\foreach \i [count=\j] in {5,...,1,6} {
	\coordinate (si) at (A.side \j);
	\path (A.center) -- (si) node[pos=1.25] {\i}; 
}
\path[thick] (A.corner 1) edge node[auto] {} (A.corner 4)
(A.corner 2) edge node[auto] {} (A.corner 6);
\end{tikzpicture}
\qquad \text{and} \qquad 
\begin{tikzpicture}[font=\scriptsize, baseline={([yshift=-.5ex]current bounding box.center)}]
\node[regular polygon, shape border rotate=90, regular polygon sides=4, minimum size=2cm, draw] at (5*2,0) (A) {};
\foreach \i [count=\j] in {3,...,1,4} {
	\coordinate (si) at (A.side \j);
	\path (A.center) -- (si) node[pos=1.25] {\i}; 
}\path[thick] (A.corner 2) edge node[auto] {} (A.corner 4);
\end{tikzpicture} 
\]
along the sides labeled ``$6$'' and ``$3$'', respectively, produces the chord diagram
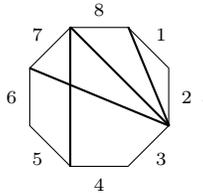
\[
\begin{tikzpicture}[font=\scriptsize, baseline={([yshift=-.5ex]current bounding box.center)}]
\node[regular polygon, shape border rotate=45, regular polygon sides=8, minimum size=2cm, draw] at (5*2,0) (A) {};
\foreach \i [count=\j] in {7,...,1,8} {
	\coordinate (si) at (A.side \j);
	\path (A.center) -- (si) node[pos=1.25] {\i}; 
}
\path[thick] (A.corner 1) edge node[auto] {} (A.corner 6)
(A.corner 1) edge node[auto] {} (A.corner 4)
(A.corner 2) edge node[auto] {} (A.corner 6)
(A.corner 6) edge node[auto] {} (A.corner 8);
\end{tikzpicture} .
\]
In this case, we grafted together two \emph{gravity} chord diagram, and the result was again a gravity chord diagram. This was no coincidence:
\begin{prop}\label{grafting}
	Grafting the distinguished side of a gravity chord diagram onto a non-distinguished side of a gravity chord diagram produces a new gravity chord diagram.
\end{prop}

\begin{proof}
	The proof amounts to checking a few cases. Suppose, for instance, that there is a configuration of chords of the form displayed in Figure \ref{grav2} after grafting. If this configuration is completely contained in one of the two polygons which have been grafted together, then this would contradict the assumption that both were gravity chord diagrams. The only other possibility is that the ``middle'' chord in this configuration is the new chord that was introduced upon grafting, since the new chord introduced after grafting cannot cross any other chord. But in this case, the other two chords were a configuration of the form displayed in Figure \ref{grav1} before grafting.
\end{proof}

\begin{defn}\label{g-def}Let $\mathsf g(n)$ denote the set of gravity chord diagrams in an $(n+1)$-gon. The grafting procedure of Proposition \ref{grafting} makes the collection $\mathsf g = \{\mathsf g(n)\}_{n \geq 2}$ into a nonsymmetric operad in the category of sets, which we call the \emph{operad of gravity chord diagrams.} 
\end{defn}

\begin{defn} We say that a chord in a chord diagram is \emph{residual} if the diagram contains no other chord which crosses it.
\end{defn}

The dual procedure to grafting two chord diagrams is \emph{cutting along a residual chord}: given a chord diagram with $d$ chords on an $n$-gon and a residual chord $c$, cutting along $c$ gives two chord diagrams: one on a $(k+1)$-gon and one on an $(n-k+1)$-gon (for some $k$), with a total of $(d-1)$ chords. 

\begin{prop}\label{cutting}
	Let $c$ be a residual chord in a gravity chord diagram on an $n$-gon. Cutting the diagram along $c$ produces two gravity chord diagrams, where the distinguished side on the polygon not containing the side labeled ``$n$'' is taken to be the side given by the chord $c$.
\end{prop}

\begin{proof}
	The proof is entirely analogous to the proof of Proposition \ref{grafting}.
\end{proof}

 \begin{defn} We say that a gravity chord diagram is \emph{prime} if it contains no residual chords. Denote the set of prime chord diagrams on an $(n+1)$-gon by \(\mathsf{p}(n) \subset \mathsf g(n)\). \end{defn}

\begin{thm}\label{chorddiagramsfree}
	The operad of gravity chord diagrams $\mathsf g$ is the free nonsymmetric operad generated by the collection $\mathsf p$ of prime chord diagrams.
\end{thm}

\begin{proof}
	Every residual chord in a gravity chord diagram gives a way of writing the diagram as an operadic composition of two smaller gravity chord diagrams, and vice versa, by Propositions \ref{grafting} and \ref{cutting}. It follows that prime chord diagrams are exactly the indecomposable elements in the operad $\mathsf g$. Moreover, every element of $\mathsf g$ can be written in a unique way as a composition of prime chord diagrams, by cutting along all possible residual chords.\end{proof}

\begin{rem}
	The preceding proof is perhaps best illustrated by an example. Figure \ref{freefigure1} shows a gravity chord diagram on the $10$-gon, and Figure \ref{freefigure2} shows the unique way of writing this gravity chord diagram as an operadic composition of prime chord diagrams. The distinguished (``output'') side of each polygon is illustrated by a thick edge.
\end{rem}

\begin{figure}[!h]
	\centering
	\begin{minipage}{.45\linewidth}
		\centering
		
		\begin{tikzpicture}[font=\scriptsize, baseline={([yshift=-.5ex]current bounding box.center)}]
		\node[regular polygon, shape border rotate=0, regular polygon sides=10, minimum size=2cm, draw] at (9.5,.75) (A) {};
		\path[thick] 
		(A.corner 1) edge node[auto] {} (A.corner 3)
		(A.corner 2) edge node[auto] {} (A.corner 5)
		(A.corner 5) edge node[auto] {} (A.corner 9)
		(A.corner 10) edge node[auto] {} (A.corner 6)
		(A.corner 10) edge node[auto] {} (A.corner 7)
		(A.corner 10) edge node[auto] {} (A.corner 5)
		(A.corner 1) edge node[auto] {} (A.corner 5);		
		\path[very thick]
		
		(A.corner 1) edge node[auto] {} (A.corner 2);
		\end{tikzpicture}
		\captionof{figure}{An element of $\mathsf g(9)$.}
		
		\label{freefigure1}
	\end{minipage}
	\hspace{.05\linewidth}
	\begin{minipage}{.45\linewidth}
		\centering

		\begin{tikzpicture}[font=\scriptsize, baseline={([yshift=-.5ex]current bounding box.center)}]
		\node[regular polygon, shape border rotate=0, regular polygon sides=10, minimum size=2cm] at (9.7,.1) (A) {};
		\path
		
		(A.corner 5) edge node[auto] {} (A.corner 4)
		(A.corner 4) edge node[auto] {} (A.corner 3)
		(A.corner 1) edge node[auto] {} (A.corner 5)
		(A.corner 3) edge node[auto] {} (A.corner 2)
		(A.corner 2) edge node[auto] {} (A.corner 1);
		\path[thick] 
		(A.corner 1) edge node[auto] {} (A.corner 3)
		(A.corner 2) edge node[auto] {} (A.corner 5);
		\path[very thick]
		(A.corner 1) edge node[auto] {} (A.corner 2);
		\node[regular polygon, shape border rotate=0, regular polygon sides=10, minimum size=2cm] at (10,0) (B) {};
		\path
		(B.corner 5) edge node[auto] {} (B.corner 1)
		(B.corner 10) edge node[auto] {} (B.corner 5)
		(B.corner 1) edge node[auto] {} (B.corner 10);
		\path[very thick]
		(B.corner 1) edge node[auto] {} (B.corner 5);
		\node[regular polygon, shape border rotate=0, regular polygon sides=10, minimum size=2cm] at (10.3,-.1) (C) {};
		\path
		(C.corner 10) edge node[auto] {} (C.corner 9)
		(C.corner 9) edge node[auto] {} (C.corner 8)
		(C.corner 8) edge node[auto] {} (C.corner 7)
		(C.corner 7) edge node[auto] {} (C.corner 6)
		(C.corner 6) edge node[auto] {} (C.corner 5)
		(C.corner 5) edge node[auto] {} (C.corner 10);
		\path[thick] 
		(C.corner 5) edge node[auto] {} (C.corner 9)
		(C.corner 10) edge node[auto] {} (C.corner 6)
		(C.corner 10) edge node[auto] {} (C.corner 7);
		\path[very thick]
		(C.corner 10) edge node[auto] {} (C.corner 5);

		\end{tikzpicture}
		
		\captionof{figure}{The decomposition of the element on the right into prime chord diagrams.}
		\label{freefigure2}
	\end{minipage}
\end{figure}

Note that the set of gravity chord diagrams is partitioned \(\mathsf{g}(n)=\coprod_{d\geq 0}\mathsf{g}(n)^d\) by the number \(d\) of chords, and the set of prime chord diagrams similarly. Define \(\mathsf{coG}(n) = \bigoplus_{d\geq 0}\Sigma^{-d-1}\Q\mathsf{g}(n)^d\). This becomes a nonsymmetric graded cooperad, by dualizing the operad structure of Definition \ref{g-def}. 
It is cofree on \(\mathsf{P}(n) = \bigoplus_{d\geq 0}\Sigma^{-d-1}\Q\mathsf{p}(n)^d\).

\begin{rem}
Notice that the collection \(\mathsf{g}\) is not a cooperad of sets, even though it linearly spans the cooperad \(\mathsf{coG}\) and the partial cocompositions take basis vectors to basis vectors (if the partial cocomposition is nonzero). In pictures this is visible in that there is no way, at the level of sets, to cocompose a gravity chord diagram along a chord not present in the diagram. In \(\mathsf{coG}\) the problem evaporates, such a cocomposition is simply set to be zero. Abstractly, the cooperad \(\mathsf{coG}\) is conilpotent and no such cooperad can be spanned by a cooperad of sets.
\end{rem}

\begin{thm}\label{thm:nscofree}
The map \(\mathsf{coG}\to \coGrav\) is an isomorphism of nonsymmetric graded cooperads. In particular, $\coGrav$ is cofree as a nonsymmetric cooperad.
\end{thm}

\begin{proof}Since we have already shown that gravity chord diagrams give a basis for the cohomology of $\M_{0,n}$, we only need to verify that this is in fact a morphism of cooperads. Modulo signs involving the convention for how to order chords in a gravity diagram, the statement amounts to showing that the Poincar\'e residue is given diagrammatically by cutting along a residual chord. To argue this we shall use Brown's presentation
 \[
  H^{\bullet}(M_{0,n}) = \wedge^\bullet C_{n} / \biggl\langle \biggl(\sum_{a\in A} \alpha_a\biggr) \biggl(\sum_{b\in B} \alpha_b\biggr)\biggr\rangle,
 \]
for \(A, B\subset C_{n}\) completely crossing subsets. We propose that the Poincar\'e residue to a stratum \(u_c=0\) is given by \(\mathrm{Res}_c = \Delta_c\circ\partial/\partial\alpha_c\), where \(\Delta_c\) is the operation of ``cutting along \(c\)''. Let \(A,B\subset C_{n}\) be a pair of completely crossing subsets and define
 \[
 R_{A,B}= \biggl(\sum_{a\in A} \alpha_a\biggr) \biggl(\sum_{b\in B} \alpha_b\biggr).
 \]
Assume first that \(c\notin A\cup B\). Then
 \(
 \frac{\partial}{\partial \alpha_c} R_{A,B}=0.
 \)
Assume conversely, without loss of generality, that \(c\in A\). Then
 \(
 \frac{\partial}{\partial \alpha_c} R_{A,B}=\sum_{b\in B} \alpha_b.
 \)
However, we note that all \(b\in B\) must then cross \(c\); and since \(\Delta_c\alpha_b=0\) if \(b\) crosses \(c\), we can conclude that, in all cases,
 \(
 \Delta_c\frac{\partial}{\partial \alpha_c} R_{A,B} = 0.
 \)
This proves that the expression is well-defined as a map on cohomology. That it equals the Poincar\'e residue is then clear since \(u_c=0\) is the equation defining the stratum and \(\alpha_c=\frac{du_c}{2i\pi u_c}\). 
\end{proof}

\begin{rem}Taking the Poincar\'e residue of Arnol'd forms \(\omega_{ij}\) is a lot more subtle --- the simple formula \(\mathrm{Res}_c = \Delta_c\circ\partial/\partial\alpha_c\), using the chord-diagrammatic description, is the reason why the basis given by the \(\{\alpha_{ij}\}\) is better suited for cooperadic computations.
\end{rem} 
\begin{rem}
	In fact, our proof shows that the integral cohomology $\{H^{\bullet-1}(M_{0,n+1},\Z)\}_{n \geq 2}$ forms a cooperad, and that this integral gravity cooperad is cofree.
\end{rem}

Recall that \(\Prim_n = H^{\bullet-2}(\M_{0,n}^{\delta})\).

\begin{cor}
The projection from the basis $\mathsf g(n)$ onto the subset $\mathsf p(n)$ cogenerates an isomorphism
 \[
 \coGrav \to \freeantipl(\Sigma\Prim)
 \]
of antiplanar cooperads, where $\freeantipl$ denotes the cofree antiplanar cooperad functor.
\end{cor}

\begin{proof}
It follows from \ref{equivalence} that if \(\coGrav\) is cofree as a nonsymmetric cooperad, then it must be cogenerated by \(\Sigma\Prim\) and, moreover, \(\Sigma\Prim\) must be isomorphic to its image in \(\coGrav\). Thus we must only argue that the projection \(H^{\bullet}(\M_{0,n})\to H^{\bullet}(\M_{0,n}^{\delta})\) is a map of planar collections, i.e., that it respects cyclic group actions. But this is clear since the cyclic action preserves residual chords, \(H^{\bullet}(\M_{0,n}^{\delta})\subset H^{\bullet}(\M_{0,n})\) is a subrepresentation of the cyclic group, and the projection is a left inverse of the inclusion.
\end{proof}

\subsection{The Lie operad}

According to a theorem of Salvatore and Tauraso \cite{salvatoretauraso}, the operad $\mathsf{Lie}$ is the linear hull of a free nonsymmetric operad in the category of sets. According to Theorems \ref{chorddiagramsfree} and \ref{thm:nscofree}, the same holds true for the gravity operad $\mathsf{Grav}$. In the highest cohomological degree, $\Grav$ is isomorphic to the suspension $\Lambda \mathsf{Lie}$ of the Lie operad; in particular, we have recovered an independent proof of Salvatore and Tauraso's result. 

In fact, there is an easy isomorphism between the operad in sets which they construct, and the suboperad of $\mathsf g$ given by gravity chord diagrams of largest degree. In \cite[227]{salvatoretauraso} it is explained how the elements of their operad can be drawn diagrammatically in terms of diagrams of arcs in a half-plane. Their diagrams are equivalent to our arc diagrams --- the only difference is that in our diagrams, the largest arc $e_{1n}$ is not allowed to occur, whereas for them it must always be present. The isomorphism between their operad and ours is then given by taking one of their diagrams, intepreting it as an arc diagram in our sense, and using the bijection between arc and chord diagrams.

According to Theorem \ref{thm:nscofree} and Theorem \ref{equivalence}, $\coGrav$ is in fact cofree as an \emph{antiplanar} cooperad. In the highest cohomological degree, $\coGrav$ is isomorphic to the suspension $\Lambda \mathsf{coLie}$ of the Lie co-operad. Desuspending and dualizing, we conclude:

\begin{cor}
	The operad $\mathsf{Lie}$ is free as a \emph{planar} operad. 
\end{cor}

This is a mild improvement on what is proven in \cite{salvatoretauraso} --- it says that not only is the $\mathsf{Lie}$ operad free, but the space of generators in arity $n$ can be given an action of the cyclic group $\Z/(n+1)\Z$, compatible in a suitable sense with the actions of the cyclic groups $\Z/(n+1)\Z$ on the spaces $\mathsf{Lie}(n)$. 

Given this, it is natural to ask if $\mathsf{Lie}$ is in fact the linear hull of a free planar operad in the category of sets. This is, however, false: one may verify (e.g.\ by using the computations of \cite[Subsection 4.4.1]{browncarrschneps}) that a generator of $\Z/6\Z$ acts on the $4$-dimensional space of generators in arity $5$ with trace $-1$, so the action cannot be given by permuting a basis. In particular, $\mathsf{Grav}$ is not the linear hull of a free planar operad in the category of sets, either.

\section{Purity}\label{sec:purity}

\subsection{The constructions of Kapranov and Keel}There are many ways one can construct the moduli space $\MM_{0,n}$ as an iterated blow-up. One is due to \cite{kapranovchow}: start with $\P^{n-3}$, and choose $n-1$ points in general linear position (such a choice is unique up to a change of coordinates). Consider the collection of subvarieties of $\P^{n-3}$ given by all projective subspaces spanned by subsets of these points. Choose a minimal element $Z$ of this collection. Replace $\P^{n-3}$ by the blow-up $\Bl_Z \P^{n-3}$ and replace the remaining elements of the collection of subvarieties by their strict transforms. Repeat this procedure until every member of the collection has been blown up. Then the result is isomorphic to $\overline M_{0,n}$.
 
A similar construction was used by \cite{keel}: start instead with $(\P^1)^{n-3}$, and consider the collection of subvarieties defined by the condition that some subset of the $n-3$ points are equal to each other, or that some subset of the points are equal to $0$, $1$ or $\infty$. An identical procedure of iteratively blowing up a minimal element of the collection and replacing the remaining ones by their strict transforms again produces $\MM_{0,n}$. 

Both constructions just described are special cases of \emph{wonderful compactifications} of an arrangement of subvarieties \cite{wonderfulcompactification}. We have taken the liberty of changing the order of blow-ups compared to the ones used by Keel and Kapranov: one of the main results of \cite{wonderfulcompactification} is that as long as certain combinatorial conditions are satisfied --- which in particular always hold when blowing up a minimal element --- then the end result of this procedure is insensitive to the order in which the subvarieties are blown up. 

\subsection{Weighted stable pointed rational curves}

The notion of \emph{weighted} stable pointed curve was introduced in \cite{hassettweighted}. One thing that Hassett realized is that both Kapranov's and Keel's results can be seen as special cases of a more general construction, which also allows modular interpretations of all the intermediate steps in the sequence of blow-ups. Before explaining this, let us recall the relevant definitions.

\begin{defn}
	A \emph{weight vector} is an $n$-tuple $\a = (a_1,\ldots,a_n)$ of numbers with $0 < a_i \leq 1$ for all $i$ and $\sum_{i=1}^{n} a_i > 2$. 
\end{defn}

\begin{defn}
	Fix a weight vector $\a$. Let $C$ be a nodal curve of arithmetic genus zero, equipped with $n$ marked points $x_1,\ldots,x_n$ contained in the smooth locus. We say that $a_i$ is the weight of $x_i$. We say that $(C,x_1,\ldots,x_n)$ is \emph{$\a$-stable} if:
	\begin{enumerate}
		\item For every irreducible component $C_0$ of $C$, the number of nodes of $C_0$ plus the sum of the weights of all markings on $C_0$ is strictly greater than $2$.
		\item For every $S \subset \{1,\ldots,n\}$ such that $x_i=x_j$ when $i,j \in S$, we have $\sum_{i\in S} a_i \leq 1$. 
	\end{enumerate}
\end{defn}

Hassett has proved that for every weight vector $\a$, there exists a fine moduli space $\MM_{0,\a}$ parametrizing $n$-pointed $\a$-stable curves of genus zero. It is a smooth projective scheme over $\Z$. When $\a = (1,\ldots,1)$ we recover the usual Deligne--Mumford compactification $\MM_{0,n}$. 

\subsection{Stratification by topological type}\label{stratification1}

Recall that the strata in the space $\MM_{0,n}$ are indexed by stable dual graphs $\Gamma$ with $n$ external half-edges (legs). The situation for the spaces $\MM_{0,\a}$ is entirely analogous. The complement of the locus of smooth curves is a strict normal crossing divisor, and the intersections of boundary strata define a stratification, which coincides with the natural stratification by topological type. The strata are again indexed by dual graphs, but with a different stability condition: if the external half-edges are assigned weights according to the weight vector $\a$, and the internal half-edges are all given weight $1$, then for any vertex the sum of the weights of the adjacent half-edges is greater than $2$. We can write the closed stratum corresponding to such a graph $\Gamma$ as $\prod_{v \in \mathrm{Vert}(\Gamma)}\MM_{0,\a(v)}$, where $\a(v)$ is the weight vector given by the weights of all half-edges adjacent to $v$. 

For example, boundary divisors correspond to subsets $S \subset \{1,\ldots,n\}$ with $\sum_{i\in S}a_i > 1$ and $\sum_{i \notin S} a_i > 1$, and each such boundary divisor is a product $\MM_{0,\a'} \times \MM_{0,\a''}$. Here $\a'$ is the weight vector obtained by deleting all elements of $S$ and adding a marking of weight $1$, and $\a''$ is the weight vector obtained by deleting elements not in $S$ and replacing them with a marking of weight $1$.

\subsection{Coincidence sets and chamber structure} \label{hassettchamber} For $S \subset \{1,\ldots,n\}$, let $\Delta_S \subset \MM_{0,\a}$ denote the subset defined by the condition that $x_i=x_j$ for $i,j \in S$. We call these loci \emph{coincidence sets}. If $\Delta_S \neq \emptyset$ then $\sum_{i \in S}a_i \leq 1$. Each coincidence set is itself a moduli space of weighted stable pointed curves: let $\a'$ be the weight vector obtained by removing all but one of the elements of $S$, and assigning the remaining element the weight $\sum_{i \in S}a_i$. Then $\Delta_S \cong \MM_{0,\a'}$. 

Let now $\a = (a_1,\ldots,a_n)$ and $\a' = (a_1',\ldots,a_n')$ be weight vectors. We write $\a' \preceq \a$ if $a_i' \leq a_i$ for all $i$. In this case, there is a natural reduction map $\MM_{0,\a} \to \MM_{0,\a'}$, given by contracting any components that may become unstable when the weights are lowered from $\a$ to $\a'$.

We say that $S \subset \{1,\ldots,n\}$ is \emph{large} if $\sum_{i\in S} a_i > 1$; otherwise, the subset is called \emph{small}. The space $\MM_{0,\a}$ only depends on the weight vector $\a$ via the information of which subsets of $\{1,\ldots,n\}$ are {large}. Geometrically, this means the following. The region 
$$ \mathscr W = \{(a_1,\ldots,a_n) \in \R^n : 0 < a_i \leq 1 \text{ for all } i, \sum_{i=1}^{n} a_i > 2\}$$ 
is subdivided into polyhedral chambers by the hyperplanes $1 = \sum_{i\in S} a_i$, for any $S \subset \{1,\ldots,n\}$. If $\a'$ and $\a$ are in the interior of the same chamber, then $\MM_{0,\a} \cong \MM_{0,\a'}$.  

Suppose instead that $\a$ and $\a'$ lie in adjacent chambers, with $\a' \preceq \a$. Then there is a unique subset $S$ which is large with respect to $\a$ but not $\a'$, namely the subset corresponding to the hyperplane separating the two chambers. Then we have
$$ \MM_{0,\a} \cong \Bl_{\Delta_S} \MM_{0,\a'}.$$
In other words, allowing the markings labeled by $S$ to ``bubble off'' onto a new component is equivalent to blowing up the coincidence set $\Delta_S$. For any subset $T$, the coincidence set $\Delta_T$ in $\MM_{0,\a}$ is the strict transform of the coincidence set $\Delta_T$ in $\MM_{0,\a'}$. 

We remark that if $\vert S \vert = 2$, then the coincidence set $\Delta_S$ is a divisor and and the blow-up in $\Delta_S$ is an isomorphism. In this case, crossing the corresponding wall changes the moduli functor (that is, the universal family over $\MM_{0,\a}$ is modified) but not the moduli space itself. Hassett calls the decomposition of $\mathscr W$ by the hyperplanes corresponding to all $S$ the \emph{fine} chamber decomposition, and the one obtained from $S$ with $\vert S \vert \geq 3$ the \emph{coarse} chamber decomposition. 

\subsection{Kapranov and Keel again} Consider first the weight vector $\a = (a,a,\ldots,a,1)$ with $a = \frac{1}{n-1} + \epsilon$, where $\epsilon > 0$ is sufficiently small. Then an $\a$-stable curve cannot have any extra components, so the moduli space $\MM_{0,\a}$ just parametrizes configurations of points on $\P^1$. Specifically, we are considering $(x_1,\ldots,x_n)$ with $x_i \neq x_n$ for all $i<n$, and such that not all $x_i$ with $i<n$ coincide. Up to a projectivity we may assume $x_1 = 0$ and $x_n = \infty$, in which case we are considering $(x_2,\ldots,x_{n-1}) \in \A^{n-2}$, not all equal to zero, up to the diagonal action of $\mathbb G_m$. We have thus found that
$$ \MM_{0,\a} \cong \P^{n-3}.$$ Under this isomorphism, the collection of coincidence sets $\Delta_S \subset \MM_{0,\a}$ becomes identified with the collection of projective subspaces spanned by all subsets of the $n-1$ points with projective coordinates
\begin{align*}
[1:0:0:\ldots:0], [0:1:0:\ldots:0], [0:0:1:\ldots:0],\ldots, [0:0:0:\ldots:1], [1:1:\ldots:1].
\end{align*}
Now suppose that we gradually increase the weights in the vector $\a$ from $(a,a,\ldots,a,1)$ to $(1,1,\ldots,1)$, in such a way that we never intersect two distinct hyperplanes  $1 = \sum_{i\in S} a_i$ simultaneously. Then by the description in the previous subsection, the moduli space $\MM_{0,\a}$ is transformed from $\P^{n-3}$ to $\MM_{0,n}$ by a sequence of blow-ups. At each step we blow up a minimal coincidence set, and each coincidence sets is the strict transform of one of the above projective subspaces in $\P^{n-3}$. Thus we see that we have exactly recovered Kapranov's construction of $\MM_{0,n}$. 

Keel's construction is recovered in exactly the same way, starting instead with the weight vector $\a = (\frac{2}{3},\frac{2}{3},\frac{2}{3},\epsilon,\ldots,\epsilon)$. Then up to a projectivity the first three markings are $0$, $1$ and $\infty$, and the remaining markings can be assigned arbitrarily. Thus $\MM_{0,\a} = (\P^1)^{n-3}$. The collection of coincidence sets is given by all subsets where some markings coincide with each other or with $0$, $1$ or $\infty$. In exactly the same way we see that gradually increasing the weights in $\a$ from $(\frac{2}{3},\frac{2}{3},\frac{2}{3},\epsilon,\ldots,\epsilon)$ to $(1,1,\ldots,1)$ recovers Keel's construction of $\MM_{0,n}$. 

Since each intermediate step in the construction is explicitly given by some space $\MM_{0,\a}$ with smaller weights, \emph{and} each blow-up center is given by some space $\MM_{0,\a}$ with fewer marked points, Hassett's construction is ideally suited for inductive arguments. 

 

 
\subsection{Weighted version of Brown's partial compactification} We now wish to define analogous spaces $\M_{0,\a}^\delta \subset \MM_{0,\a}$ for arbitrary weight vectors $\a$, generalizing $M_{0,n}^\delta \subset \MM_{0,n}$. For this we will need a dihedral structure $\delta$ on $\{1,\ldots,n\}$, which we continue to assume is the standard one. We say that a subset $I \subset \{1,\ldots,n\}$ is an \emph{interval} if it is so for this dihedral structure. For instance, $\{2,3,4,5\}$ is an interval, but so is also $\{n-2,n-1,n,1,2\}$.

One could define $\M_{0,\a}^\delta$ simply as the Zariski open subset parametrizing those curves whose dual graph is compatible with the dihedral structure, but for our purposes this does not turn out to be the right definition. 

The space $\MM_{0,\a}$ only depended on the weight vector $\a$ via the collection of subsets $S \subset \{1,\ldots,n\}$ such that $\sum_{i\in S}a_i > 1$. Similarly, we want the space $\M_{0,\a}^\delta$ to depend only on the collection of \emph{intervals} $I \subset \{1,\ldots,n\}$ such that $\sum_{i\in I}a_i > 1$. As before, we say that such an interval is \emph{large}, and $I$ is said to be \emph{small} otherwise.

\begin{defn}
	We define $\M_{0,\a}^\delta$ to be the Zariski open subset of $\MM_{0,\a}$ parametrizing those weighted stable $n$-pointed curves which are compatible with the given dihedral structure, and such that if $x_i = x_j$ for $i,j \in S \subset \{1,\ldots,n\}$, then $S$ is contained in a small interval. 
\end{defn} 
 
Clearly $\M_{0,\a}^\delta = M_{0,n}^\delta$ if $\a = (1,1,\ldots,1)$.  
 
\subsection{Stratification of $\M_{0,\a}^\delta$ by topological type} 
The space $\M_{0,n}^\delta$ has a stratification by topological type, whose strata correspond bijectively to pairwise disjoint collections of chords in the $n$-gon. Such a collection of chords gives rise to a tiling of the $n$-gon by smaller polygons. The closure of such a stratum is a product of moduli spaces $\M_{0,n_i}^\delta$, one for each polygon in the tiling, where $n_i$ is the number of edges of the corresponding polygon. 

The space $\M_{0,\a}^\delta$ also has a stratification by topological type, whose strata correspond to collections of chords as above satisfying the following additional stability condition: if each chord is given weight $1$, and the $i$th edge of the $n$-gon is given weight $a_i$, then the sum of all weights along the edges of each smaller polygon is greater than $2$. Again the closure of such a stratum is a product of smaller moduli spaces of the form $\M_{0,\a}^\delta$, one for each polygon in the tiling, with weight vector and dihedral structure specified by the weights along the edges of each polygon. An example is illustrated in Figures \ref{strata1} and \ref{strata2}.

As in the case of $\M_{0,n}^\delta$ this stratification is given by a strict normal crossing divisor in $\M_{0,\a}^\delta$; each divisor corresponds to a single chord in the $n$-gon, dividing the weight vector into two large intervals.

\begin{figure}
	\centering
	\begin{minipage}{.45\linewidth}
		\centering
\begin{tikzpicture}[font=\scriptsize, baseline={([yshift=-.5ex]current bounding box.center)}]
\node[regular polygon, shape border rotate=45, regular polygon sides=8, minimum size=2cm, draw] at (5*2,0) (A) {};
\foreach \i in {1,...,8} {
	\coordinate (ci) at (A.corner \i);
	\coordinate (si) at (A.side \i);
	\node at (ci) {};
	\path (A.center) -- (si) node[pos=1.25] {$a_\i$}; 
}
\path 
(A.corner 1) edge node[auto] {} (A.corner 3)
(A.corner 8) edge node[auto] {} (A.corner 4)
(A.corner 6) edge node[auto] {} (A.corner 4);
\end{tikzpicture} 
		\captionof{figure}{A stratum inside $\M_{0,\a}^\delta$, $\a = (a_1,\ldots,a_8)$. The stability condition is equivalent to $a_1+a_2 >1$, $a_4+a_5 > 1$.}
		\label{strata1}
	\end{minipage}
	\hspace{.05\linewidth}
	\begin{minipage}{.45\linewidth}
		\centering

		\begin{tikzpicture}[font=\scriptsize, baseline={([yshift=-.5ex]current bounding box.center)}]
		\node[regular polygon, shape border rotate=45, regular polygon sides=8, minimum size=2cm] at (9,0) (A) {};
		\path
		(A.corner 1) edge node[auto] {$1$} (A.corner 3)
		(A.corner 3) edge node[auto] {$a_2$} (A.corner 2)
		(A.corner 2) edge node[auto] {$a_1$} (A.corner 1);
		\node[regular polygon, shape border rotate=45, regular polygon sides=8, minimum size=2cm] at (10,0) (B) {};
		\path
		(B.corner 4) edge node[auto] {$a_3$} (B.corner 3)
		(B.corner 3) edge node[auto] {$1$} (B.corner 1)
		(B.corner 8) edge node[auto] {$1$} (B.corner 4)
		(B.corner 1) edge node[auto] {$a_8$} (B.corner 8);
		\node[regular polygon, shape border rotate=45, regular polygon sides=8, minimum size=2cm] at (11,0) (C) {};
		\path
		(C.corner 8) edge node[auto] {$a_7$} (C.corner 7)
		(C.corner 7) edge node[auto] {$a_6$} (C.corner 6)
		(C.corner 6) edge node[auto] {$1$} (C.corner 4)
		(C.corner 4) edge node[auto] {$1$} (C.corner 8);
		\node[regular polygon, shape border rotate=45, regular polygon sides=8, minimum size=2cm] at (12,-0.3) (D) {};
		\path
		(D.corner 6) edge node[auto] {$a_5$} (D.corner 5)
		(D.corner 5) edge node[auto] {$a_4$} (D.corner 4)
		(D.corner 4) edge node[auto] {$1$} (D.corner 6);
		
		\end{tikzpicture}

\captionof{figure}{The stratum depicted on the left is the product of the moduli spaces corresponding to these four polygons.}
\label{strata2}
\end{minipage}
\end{figure}

\subsection{Coincidence sets in the dihedral case} For $S \subset \{1,\ldots,n\}$ we continue to denote the coincidence set by $\Delta_S \subset \M_{0,\a}^\delta$. 

For the remainder of this subsection, we fix a subset $S$ such that $\Delta_S \neq \emptyset$. Then there is a minimal small interval containing $S$; let us denote it $I$. 

\begin{lem}\label{lemma1}
	The topological type of a point of $\Delta_S$ is given by a configuration of chords disjoint from $I$.  
\end{lem}

\begin{proof}
	The topological type cannot contain a chord contained in $I$, since $I$ is small. It cannot contain a chord that starts in $I$ and ends outside it, either, since such a chord separates the endpoints of the interval $I$, but the markings corresponding to the endpoints need to coincide along the locus $\Delta_S$.
\end{proof}

Choose an arbitrary element $s \in S$. Let $\a'$ be the weight vector with $n + 1 - \vert S \vert$ elements obtained by removing all elements of $S \setminus \{s\}$, and assigning the weight $\sum_{i \in S} a_i$ to $s$. (We formulate the procedure in this way to emphasize that the dihedral structure on $\a'$ depends on the choice of element $s \in S$.) In the situation of Hassett's spaces, we had that $\Delta_S \subset \MM_{0,\a}$ was isomorphic to $\MM_{0,\a'}$. For the dihedral spaces, this statement needs to be modified; we have instead the following lemma: 

\begin{lem}\label{coincidenceset}
	Let $I'$ be the small interval in the $(n+1-\vert S \vert)$-gon just defined, obtained by deleting the elements of $S \setminus \{s\}$ from $I$. Then $\Delta_S \subset M_{0,\a}^\delta$ is isomorphic to the open subset of $\M_{0,\a'}^\delta$ which is the complement of all boundary divisors corresponding to chords that meet $I'$.  
\end{lem}

\begin{proof}
	By Lemma \ref{lemma1}, the image of the natural map $\Delta_S \to \M_{0,\a'}^\delta$ is contained in this open subset. Conversely, it is not hard to see that this map has a well defined inverse given by adding new markings on top of $x_s$ away from said boundary divisors. 
\end{proof}

\begin{lem}\label{retract}
	The inclusion $i \colon \Delta_I \hookrightarrow \Delta_S$ is a retract; that is, there is a map $r$ in the opposite direction with $r \circ i = \id$. 
\end{lem}

\begin{proof}
	The map is the only natural one: it sets all markings $x_i$ for $i \in I$ equal to $x_j$ for $j \in S$. We should verify that this is well defined. By Lemma \ref{lemma1}, this does not affect the topological type of the curve, and in particular will not cause any component to become unstable. Moreover, if this causes some collection of points $\{x_i\}_{i \in T}$ to coincide, then the markings indexed by $(T \setminus I) \cup S$ must have coincided already before applying $r$. Thus $(T \setminus I) \cup S$ is contained in a short interval, and this interval must contain all of $T$. 
\end{proof}

\begin{rem}
	It is not true in general that if $S \subset T$, then $\Delta_T \hookrightarrow \Delta_S$ is a retract. 
\end{rem}

\subsection{Wall-crossing for $\M_{0,\a}^\delta$} \label{wallcrossing} Suppose that $\a' \preceq \a$. Then the reduction map $\MM_{0,\a} \to \MM_{0,\a'}$ maps the open subset $\M_{0,\a}^\delta$ into $\M_{0,\a'}^\delta$, so we get well defined reduction maps also between the dihedral spaces.

Consider the region $\mathscr W$ from Subsection \ref{hassettchamber}, parametrizing all possible weight vectors. It 
can be subdivided into polyhedral chambers by the hyperplanes $1 = \sum_{i\in I}a_i$ where $I \subset \{1,\ldots,n\}$ is an interval, giving rise to a coarser chamber decomposition than the one considered in the previous section. If $\a$ and $\a'$ lie in the interior of the same chamber with respect to this coarser decomposition, then $M_{0,\a}^\delta \cong M_{0,\a'}^\delta$.

Suppose that $\a$ and $\a'$ lie in adjacent chambers and that $\a' \prec \a$. We wish to understand the relationship between the spaces $\M_{0,\a}^\delta$ and $\M_{0,\a'}^\delta$. There will be a unique interval $I$ which is small with respect to $\a'$ and large with respect to $\a$. Suppose that $i$ and $j$ are the endpoints of the interval. What happens is that  the reduction map $\M_{0,\a}^\delta \to \M_{0,\a'}^\delta$
admits a factorization:
$$ \M_{0,\a}^\delta \hookrightarrow \Bl_{\Delta_I} \M_{0,\a'}^\delta \to \M_{0,\a'}^\delta.$$
The second map is the blow-up along the coincidence set $\Delta_I$. The first map is an open immersion, which is the inclusion of the complement of the strict transform of the divisor $\Delta_{\{i,j\}}$. 

These statements follow from the corresponding ones for Hassett's spaces $\MM_{0,\a}$. Indeed, Hassett's spaces are modified by blowing up $\Delta_S$ when crossing the wall $1 = \sum_{i\in S}a_i$. If $S = I$ is an interval, then $ \Bl_{\Delta_I} \M_{0,\a'}^\delta$ is naturally an open subset of $\Bl_{\Delta_I} \MM_{0,\a'} = \MM_{0,\a}$. The only difference between this open subset and $\M_{0,\a}^\delta$ is that $x_i$ and $x_j$ are allowed to coincide in $ \Bl_{\Delta_I} \M_{0,\a'}^\delta$. Thus removing the strict transform of $\Delta_{\{i,j\}}$ produces $\M_{0,\a}^\delta$. (We remark that this is true also in case $\{i,j\} = I$: in this case $\Delta_I$ is a divisor and blowing up $\Delta_I$ is an isomorphism. Then we remove the strict transform of $\Delta_{\{i,j\}}$, which is empty.)

\subsection{An inductive construction of $M_{0,\a}^\delta$}

The results proven thus far in this section can be used to give an explicit procedure for constructing the moduli space $\M_{0,n}^\delta$ from the affine space $\A^{n-3}$ by repeatedly blowing up a subvariety and then removing the strict transform of a divisor containing the blow-up center. More precisely, we have the following theorem:

\begin{thm}\label{construction}
	Let $\a = (a_1,\ldots,a_n)$ be a weight vector with $a_n=1$. As in Section \ref{sectionarc}, identify $\A^{n-3}$ with the space
	$$ X = \{(z_1,\ldots,z_{n-1}) \in \A^{n-1} : z_1=0, z_{n-1}=1\}.$$
	For every large interval $I = \{i,i+1,\ldots,j-1,j\} \subsetneq \{1,\ldots,n-1\}$ with respect to this weight vector, let $Z_I$ be the affine subspace $\{z_i=z_{i+1}=\ldots=z_j\}$ of $\A^{n-3}$, and $Y_I$ the hyperplane $\{z_i=z_j\}$. 
	
	Iteratively carry out the following procedure:
	\begin{enumerate}
		\item Let $Z_J$ be a minimal element of the collection $\{Z_I\}$.
		\item Replace $X$ with $\mathrm{Bl}_{Z_J}(X) \setminus \widetilde Y_J$, the blow-up of $X$ in $Z_J$ minus the strict transform of the divisor $Y_J$. Moreover, replace each $Z_I$ and $Y_I$ with their strict transforms in this blow-up.
		\item Repeat steps \emph{(1)} and \emph{(2)} until all elements $Z_I$ have been blown up.
	\end{enumerate}
	The end result is isomorphic to the space $\M_{0,\a}^\delta$. 
\end{thm}

\begin{proof}
	Let us first make the observation that if $a_n=1$, then every interval containing $n$ is automatically large,  and the interval $\{1,\ldots,n-1\}$ is also large (since the total weight is at least two). Thus the set of large intervals 
	$I \subsetneq \{1,\ldots,n-1\}$ is exactly the set of intervals which are not automatically large because of the fact that $a_n=1$.
	
	We prove the result by induction on the number of large intervals $I \subsetneq \{1,\ldots,n-1\}$. The base case is if there are no such large intervals, which happens e.g.\ for the weight vector 
	$$ \left(a,a,\ldots,a,1\right)$$
	where $a = \frac{1}{n-1} + \epsilon$. In this case (by the observation in the previous paragraph), $\M_{0,\a}^\delta$ parametrizes $n$ points $(x_1,\ldots,x_n)$ on $\P^1$ such that $x_i \neq x_n$ for any $i < n$, and $x_1 \neq x_{n-1}$. Up to a projectivity we can set $x_1 = 0$, $x_{n-1}=1$ and $x_n = \infty$, and the moduli space is equal to $\A^{n-3}$. This proves the base case. Moreover, we make the observation that if $I = \{i,i+1,\ldots,j\}$, then the subvariety $Z_I \in \A^{n-3}$ becomes identified with the subvariety $\Delta_I \subset \M_{0,\a}^\delta$, and $Y_I$ the subvariety $\Delta_{\{i,j\}}$. 
	
	For the induction step, suppose that $\a$ is a weight vector, and that $\a' \prec \a$ is in an adjacent chamber. Then there is a unique interval $I$ with endpoints $\{i,j\}$ which is large with respect to $\a$ but not $\a'$. As described in Subsection \ref{wallcrossing} $\M_{0,\a}^\delta$ is obtained from $M_{0,\a'}^\delta$ by blowing up $\Delta_I$ and removing the strict transform of $\Delta_{\{i,j\}}$. Moreover, the strict transform of a coincidence set is a coincidence set, so $\Delta_I$ is the iterated strict transform of $Z_I \subset \A^{n-3}$ and $\Delta_{\{i,j\}}$ is the iterated strict transform of $Y_I \subset \A^{n-3}$. The result follows.	
\end{proof}

Already the first non-trivial example $n=5$ is very instructive. In this case, our description says that $\M_{0,5}^\delta$ is isomorphic to the variety obtained by blowing up $\A^2$ in the points $(0,0)$ and $(1,1)$, and removing the strict transforms of the two lines $y=0$ and $x=1$. See Figure \ref{img1}.
	
	\begin{figure}[h]
		\centering
		\begin{minipage}{.45\linewidth}
			\centering
			\begin{tikzpicture}[scale=0.55]
			
			\node (v4) at (-3.5,4) {};
			\node (v3) at (-3.5,11) {};
			\node (v1) at (-5,9.5) {};
			\node (v2) at (2,9.5) {};
			\draw  (v1) edge (v2);
			\draw  (v3) edge (v4);
			\node at (0.5,9.5) {};
			\node (v5) at (0.5,9.5) {};
			\node (v6) at (-3.5,5.5) {};
			\filldraw  (v5) ellipse (0.1 and 0.1);
			\filldraw  (v6) ellipse (0.1 and 0.1);
			\end{tikzpicture}
			
			\captionof{figure}{$\M_{0,5}^\delta$ is obtained from $\A^2$ by blowing up the two thick marked points and removing the strict transform of the two lines.}
			\label{img1}
		\end{minipage}
		\hspace{.05\linewidth}
		\begin{minipage}{.45\linewidth}
			\centering
			\begin{tikzpicture}[scale=0.45]
			\begin{scope}[yscale=1,xscale=-1, shift={(0,2.7)}]
			\draw (-3,-4) node (v1) {} -- (0,1.75) node (v3) {} -- (0,-2) node (v5) {} -- (intersection of v1--v3 and v1--v5);
			
			\begin{scope}[shift={(-5,1.5)}]
			
			\draw (-3,-4) node (v6) {} -- (0,1.75) node (v2) {};
			\draw[dashed] (v2) -- (0,-2) node (v4) {} -- (v6);
			\end{scope}
			\draw (intersection of v2--v6 and v2--v4) edge (intersection of v3--v5 and v3--v1);
			\draw[dashed]  (v4) edge (v5);
			\draw  (intersection of v1--v3 and v1--v5) edge (intersection of v6--v3 and v6--v5);
			\node (v7) at (-9.5,0.25) {};
			\node (v9) at (-3,-5.5) {};
			\node (v8) at (-3,5.2) {};
			\node (c) at (-3,2.75) {};
			\node (a) at (-5.5,4.75) {};
			\node (b) at (-2.5,-5.5) {};
			\draw[line width=1.5pt]  (a) edge (b);
			\node (v10) at (-2.45,4.95) {};
			\draw[line width=1.5pt]  (v10) edge (v7);

			\draw[line width=4pt,color=white]  (v8) edge (c);
			\draw[line width=1.5pt]  (v8) edge (v9);
			\end{scope}
			\end{tikzpicture}

			\captionof{figure}{Diagram illustrating the construction of $\M_{0,6}^\delta$ from $\A^3$.}
			\label{img2}
		\end{minipage}
	\end{figure}
	
	When $n=6$, the construction is illustrated in Figure \ref{img2}. Here we will need to perform five blow-ups, in which the blow-up centers are given by the three thick lines and their two intersection points, and remove the strict transforms of five planes, which are drawn as the planes bounding the solid prism in the figure.  We begin by blowing up the points $(0,0,0)$ and $(1,1,1)$ in $\A^3$, which are the two intersection points of the thick lines, and then removing the strict transforms of the planes $z=0$ and $x=1$, which are the two rectangular backsides of the prism. Then we blow up the strict transforms of the remaining three lines ($x=y=0$, $x=y=z$ and $y=z=1$), and remove the iterated strict transforms of the remaining three planes ($y=0$, $x=z$ and $y=1$).

\subsection{Proof of purity} \label{sec:proof}

We are almost ready to prove the assertions about the mixed Hodge structure of $M_{0,n}^\delta$, but we shall need two further cohomological lemmas. The first of the two contains the heart of the whole argument. 
 
\begin{lem} \label{mainlem} Let $Z \subset Y \subset X$ be a chain of smooth closed subvarieties, where $Y$ has codimension $1$ in $X$. Suppose that $H^k(X)$ and $H^k(Y)$ are pure of weight $2k$ for all $k$, and that $H^\bullet(Y) \to H^\bullet(Z)$ is onto. Let $\widetilde X = \Bl_Z X$, and let $\widetilde Y$ be the strict transform of $Y$. Then $H^k(\widetilde X \setminus \widetilde Y)$ is pure of weight $2k$ for all $k$. \end{lem}

\begin{proof}
	If $d$ denotes the codimension of $Z$, then by the blow-up formula we have 
	$$ H^k(\widetilde X) = H^k(X) \oplus H^{k-2}(Z)(-1) \oplus H^{k-4}(Z)(-2) \oplus \ldots \oplus H^{k-2d}(Z)(-d)$$
	and
	$$ H^k(\widetilde Y) = H^k(Y) \oplus H^{k-2}(Z)(-1) \oplus \ldots \oplus H^{k-2d+2}(Z)(-d+1).$$ There is also a long exact sequence
	$$  \ldots \to H^{k-1}(\widetilde X \setminus \widetilde Y) \to H^{k-2}(\widetilde Y)(-1) \to H^k(\widetilde X) \to H^k(\widetilde X \setminus \widetilde Y) \to \ldots $$
	Consider the Gysin map $H^{k-2}(\widetilde Y)(-1) \to H^k(\widetilde X) $. Each summand in the direct sum decomposition above has different weight, so compatibility of weights forces the Gysin map to be the direct sum of the restriction map $H^{k-2}(Y)(-1) \to H^{k-2}(Z)(-1)$ (which we assumed surjective) and the identity maps of $H^{k-2i}(Z)(-i)$. This implies that the long exact sequence splits up into a sum of exact sequences of the form
	$$ 0 \to H^k(X) \to H^k(\widetilde X \setminus \widetilde Y) \to H^{k-1}(Y)(-1) \to H^{k-1}(Z)(-1) \to 0.$$
In particular it follows that $H^k(\widetilde X \setminus \widetilde Y)$ is pure of weight $2k$. \end{proof}

\begin{lem}\label{leray}
Suppose that $X$ is a smooth variety and $D \subset X$ is a strict normal crossing divisor, $D = D_1 \cup \ldots \cup D_k$. For $I \subseteq \{1,\ldots, k\}$ we let $D_I = \bigcap_{i\in I} D_i$, including $D_\emptyset = X$. Suppose that $H^k(D_I)$ is pure of weight $2k$ for all $I$ and $k$. Then there exists an isomorphism
\begin{equation*}
H^k(X\setminus D) \cong \bigoplus_{I = \{i_1,\ldots,i_q\}}H^{k-q}(D_I)(-q).
\end{equation*} 
In particular, also $H^k(X \setminus D)$ is  pure of weight $2k$.
\end{lem}

\begin{proof}
	The Leray spectral sequence of the embedding of \(X\setminus D\) in \(X\) reads
	$$ E_2^{pq} = \bigoplus_{\vert I \vert =q }H^p(D_I)(-q) \implies H^{p+q}(X \setminus D).$$
	The hypothesis says that $E_2^{pq}$ is pure of weight $2(p+q)$, so compatibility with weights forces the spectral sequence to degenerate immediately, and the claimed isomorphism follows.  
\end{proof}

Let us now turn to the proof of Theorem \ref{thmthree} from the introduction. 

\begin{thm}\label{thm:purity}
	Let $\a$ be a weight vector such that at least one marking has weight $1$. Then $H^k(\M_{0,\a}^\delta)$ is pure of weight $2k$ for all $k$. In particular, this holds for the moduli space $M_{0,n}^\delta$.
\end{thm}
\begin{proof}
We are going to prove this by induction on the number of large intervals, using the inductive construction of $\M_{0,\a}^\delta$ from $\A^{n-3}$ described in Theorem \ref{construction}. The base case for the induction is thus $\A^{n-3}$ itself, which clearly has $H^k$ pure of weight $2k$ for all $k$.

For the induction step, suppose that $\a' \prec \a$ are weight vectors, with a unique interval $I$ with endpoints $\{i,j\}$ which is large with respect to $\a$ but not $\a'$. By induction, $\M_{0,\a'}^\delta$ has $H^k$ pure of weight $2k$. We wish to prove the same thing for $\M_{0,\a}^\delta$, which is the blow-up of $\M_{0,\a'}^\delta$ in $\Delta_I$ minus the strict transform of $\Delta_{\{i,j\}}$. By Lemma \ref{mainlem}, we are done if we can prove: (i) that $H^k(\Delta_{\{i,j\}})$ is pure of weight $2k$ for all $k$, and (ii) that $H^\bullet(\Delta_{\{i,j\}}) \to H^\bullet(\Delta_I)$ is a surjection.

For (i), let $\mathscr B$ be the weight vector given by deleting $a_j$ and replacing $a_i$ by the sum $a_i+a_j$. By Lemma \ref{coincidenceset}, $\Delta_{\{i,j\}}$ is isomorphic to the complement of a union of boundary divisors in $\M_{0,\mathscr B}^\delta$. By induction on $n$, $\M_{0,\mathscr B}^\delta$ and all intersections of boundary divisors on it have $H^k$ pure of weight $2k$ (note that all smaller moduli spaces involved will have a marking of weight $1$). We conclude from Lemma \ref{leray} that the same is true for $\Delta_{\{i,j\}}$. 

For (ii), the inclusion $\Delta_I \hookrightarrow \Delta_{\{i,j\}}$ is a retract by Lemma \ref{retract}, which implies in particular that the restriction map in cohomology is surjective. This concludes the proof.
\end{proof}

\begin{rem}
	It seems plausible that the same result holds for all the moduli spaces $\M_{0,\a}^\delta$ --- that is, also those which do not have a marking of weight $1$ --- but we do not have a proof of this fact. One would need to verify that $H^k(\M_{0,\a}^\delta)$ is pure of weight $2k$ for all collections of weights with $\sum_{i=1}^n a_i$ arbitrarily close to $2$. 
\end{rem}

\appendix
\section{Preliminaries on (planar) operads}\label{sec:operads}

In this appendix we review some necessary background material on operads, both for completeness and to fix notation. 
All dg (co)operads are assumed to be (co)augmented. This allows one to discard with the distinction between dg (co)operads and dg pseudo-(co)operads, and we will accordingly drop the qualifying prefix ``pseudo'' in front of operads and cooperads outside this section. We assume all dg cooperads to be conilpotent. We otherwise follow the conventions concerning operads adopted in \cite{lodayvallette}. A notable exception is the notion of \textit{planar} (co)operads, which to our knowledge has only one real precedent in the literature, see~\cite[Section 3]{menichi}, though many have remarked on the basic idea. The idea for the concept is simple enough: just like cyclic operads are based on trees, operads on rooted trees, and nonsymmetric operads on planar rooted trees, our notion of planar operads is based on planar (non-rooted) trees. Given the established terminology in the field, planar operads should perhaps be called nonsymmetric cyclic operads. For those who are already familiar with operads, the geodesic definition of a planar operad \(\mathsf{O}\) is as follows.

\subsection{The brief definition}

\begin{defn}
	A \emph{planar pseudo-operad} is a nonsymmetric pseudo-operad \(\mathsf{O}\) (in some cocomplete symmetric monoidal category) where each component \(\mathsf{O}(n)\) has an action of the cyclic group \(\Z/(n+1)\Z\), satisfying the following compatibility relations: if \(\tau:\mathsf{O}(n)\to\mathsf{O}(n)\) is the right action of the generating cycle \((n+1\,1\dots n)\), then
	\[
	(\phi\circ_1\psi)\tau = \psi\tau\circ_n \phi\tau, \;\; \forall \phi\in\mathsf{O}(m),\,\psi\in\mathsf{O}(n),\,m,n\geq 2,
	\]
	while
	\[
	(\phi\circ_i\psi)\tau = \phi\tau \circ_{i-1} \psi\tau, \;\; \forall \phi\in\mathsf{O}(m),\,\psi\in\mathsf{O}(n),\,m,n\geq 2, 2\leq i\leq m.
	\]
	A \emph{planar operad} is a nonsymmetric operad with a compatible collection of cyclic group actions, as above, and additionally satisfying that the generator of \(\Z/2\Z\) maps the operad unit \(1\in\mathsf{O}(1)\) to itself. 
	Suitably reversing arrows defines the notion of \emph{planar (pseudo-)cooperads}.
\end{defn}

The definition has an important sibling notion in the special case when the ambient monoidal category is the category of dg vector spaces over a field $\mathbb K$, with Koszul sign rules. This is the notion of an \emph{antiplanar dg operad}, or what we might have called nonsymmetric anticyclic dg operads --- they are to anticyclic dg operads what planar dg operads are to cyclic dg operads. 

In the next section we give a more thorough treatment, not taking the definition of nonsymmetric operads for granted. The reader is advised to refer to this portion of the paper only as needed. 

\subsection{The free planar operad functor}

A \emph{stable labeled planar tree} is a tree graph, where every vertex \(v\) has a specified cyclic order on the set of adjacent half-edges, every vertex has valency \(\geq 3\), and the set of legs is numbered by an order-preserving bijection with the cyclically ordered set \(\{1,\dots,n\}\), for some \(n\geq 3\). An isomorphism of stable labeled planar graphs is an isomorphism of the underlying graphs that respects all extra structure. With these conventions stable labeled planar trees form a groupoid \(\mathsf{PT}\). Note that it decomposes into subgroupoids \(\mathsf{PT}_n\) of trees with \(n\) legs.

Fix a cocomplete symmetric monoidal category \(\mathsf{V}\), such that $- \otimes -$ is cocontinuous in both variables. A \emph{planar collection} in \(\mathsf{V}\) is an indexed family \(\{K_n\mid n\geq 3\}\) of objects in \(\mathsf{V}\), such that \(K_n\) is a representation of the cyclic group \(\Z/n\Z\). Such collections form a category. Moreover, every planar collection \(K\) defines a functor
\[
K[\,]: \mathsf{PT} \to \mathsf{V}
\]
on the category of stable labeled planar trees and their isomorphisms, via
\[
K[T] = \bigotimes_{v\in \mathrm{Vert}(T)} K_{n(v)}.
\]
Above \(n(v)\) is the number of half-edges adjacent to the vertex. To be precise, instead of \(K_{n(v)}\) one should write
\[
\biggl( \bigoplus_{F_v\cong \{1,\ldots,n\}} K_{n(v)} \biggr)_{\Z/n\Z},
\]
a sum over order-preserving bijections between $F_v$ --- the cyclically ordered set of half-edges adjacent to $v$ --- and a standard cyclically ordered set. This can be used to define an endofunctor \(\freepl\) on the category of planar collections by
\[
\freepl(K)_n = \mathrm{colim}\bigl(\mathsf{PT}_n \xrightarrow{K[\,]}\mathsf{V}\bigr).
\]
\newcommand{\Tree}{\mathrm{PTree}}

If we let $\Tree_n$ denote the set of isomorphism classes of stable planar trees with $n$ legs, then one may write somewhat informally
\[
\freepl(K)_n = \bigoplus_{T \in \Tree_n} K[T] = \bigoplus_{T \in \Tree_n} \bigotimes_{v \in \mathrm{Vert}(T)}K_{n(v)}.
\]
\begin{defn}
	We call \(\freepl\) the \emph{free planar operad functor}.
\end{defn}

\subsection{The definition of planar operads}

Assume that \(T\) is a stable, labeled planar tree and that for every vertex \(u\in \mathrm{Vert}(T)\) of \(T\) we are given a stable planar tree \(T_u\) whose legs are labeled by the half-edges adjacent to $u$. Then we can build a tree \(T'\) that contains each \(T_u\) as a subtree and has the property that contracting all the subtrees \(T_u\) of \(T'\) produces the original tree \(T\). In particular, \[ \mathrm{Vert}(T')=\coprod_{u\in \mathrm{Vert}(T)} \mathrm{Vert}(T_u),\] giving a canonical isomorphism
\[
\bigotimes_{u\in \mathrm{Vert}(T)}\bigotimes_{v\in \mathrm{Vert}(T_u)} K_{n(v)} \cong \bigotimes_{w\in \mathrm{Vert}(T')} K_{n(w)}.
\]
Now note that 
\[ (\freepl \circ \freepl)(K)_n  = \bigoplus_{T \in \Tree_n} \bigotimes_{u \in \mathrm{Vert}(T)} \bigoplus_{T_u \in \Tree_{n(u)}} \bigotimes_{v \in \mathrm{Vert}(T_u)} K_{n(v)}, \]
and using our assumption that $\otimes$ is cocontinuous we may rewrite this as a the direct sum of $\bigotimes_{u\in \mathrm{Vert}(T)}\bigotimes_{v\in \mathrm{Vert}(T_u)} K_{n(v)}$, where the sum is taken over all $T \in \Tree_n$ and all tuples $(T_{u})_{u \in \mathrm{Vert}(T)}$. 
Taking the summand corresponding to $T$ and $(T_u)$ to the summand corresponding to the tree $T'$ defines a natural transformation \(\freepl \circ \freepl\to \freepl\). The inclusion of trees with one vertex gives a natural transformation \(\id\to \freepl\). Together these two natural transformations give the free planar operad functor the structure of a monad.

\begin{defn}
	A \emph{planar (pseudo-)operad} in \(\mathsf{V}\) is an algebra for the free planar operad monad. A \emph{morphism} of planar (pseudo-)operads is a morphism of algebras for the free planar operad monad.
\end{defn}

A planar operad is determined by a planar collection \(\mathsf{O}\) and a family of \emph{composition} morphisms
\[
\circ_i^j : \mathsf{O}_n\otimes\mathsf{O}_k\to\mathsf{O}_{n+k-2},
\]
parametrized by \(1\leq i\leq n\), \(1\leq j\leq k\), satisfying certain associativity and equivariance conditions. These morphisms arise as follows. Let \(t_n\) be the tree with a single vertex and \(n\) legs. Graft the \(i\)th leg of the tree \(t_n\) to the \(j\)th leg of \(t_k\), to obtain a tree \(t_n\circ_i^j t_k\): the composition of \(\mathsf{O}\) is the morphism
\[
\mathsf{O}[t_n\circ_i^j t_k] \to \mathsf{O}[t_{n+k-2}]
\]
defined by the algebra structure \(\freepl(\mathsf{O})\to\mathsf{O}\). In fact, only the operations
\[
\circ^{n+1}_i: \mathsf{O}_{m+1}\otimes\mathsf{O}_{n+1} \to \mathsf{O}_{m+n},\;\; 1\leq i\leq n,
\]
suffice. We could have defined a planar (pseudo-)operad as a stable collection \(\mathsf{O}\) such that the collection \(\{\mathsf{O}(n)=\mathsf{O}_{n+1}\}_{n\geq 2}\) together with the operations \(\circ_i=\circ^{n+1}_i\) is a nonsymmetric (pseudo-)operad, satisfying some compatibility with the cyclic group actions, as in the beginning of this section.

\begin{rem}
	One can phrase the theory of planar operads in dual language, using dissected planar polygons in place of planar trees. Briefly, the dual of a \(n\)-legged planar corolla \(t_n\) is a planar \(n\)-gon \(\pi_n\). Let \(\chi_r(n)\) denote the set of dissections \(D\) of \(\pi_n\) into \(r+1\) smaller polygons (i.e., \(D=\{c_1,\dots,c_r\}\) is a collection of \(r\) pairwise nonintersecting chords on \(\pi_n\)). For instance, $\chi_0(n)= \{\pi_n\}$, and \(\chi_1(n)\) denotes the set of chords of \(\pi_n\). Set \(\chi(n)=\coprod_{r\geq 0} \chi_r(n)\). The free planar operad on a collection \(K\) can equally well be regarded as a colimit
	\[
	\freepl(K)_n = \coprod_{D\in \chi(n)} K[D].
	\]
\end{rem}

\subsection{Antiplanar operads}

Let us specialize now to the case when \(\mathsf{V}\) is the category of dg vector spaces, with Koszul sign rules. We can then define a slight variation of the free planar operad functor, as follows. Define
\[
\mathrm{det}\otimes K[\,] :T \mapsto \mathrm{det}(\mathrm{Vert}(T))\otimes K[T],
\]
where the determinant \(\mathrm{det}(S)\) of a finite set \(S\) is defined to be the top exterior power \(\wedge^{\#S}\mathbb{K}^S\), placed in degree zero. The formula
\[
\freeantipl(K)_n = \mathrm{colim}\bigl(\mathsf{PT}_n\xrightarrow{\mathrm{det}\otimes K[\,]}\mathsf{V}\bigr)
\]
again defines a monad. 

\begin{defn}
	The \emph{free antiplanar monad} is the functor \(\freeantipl\). The algebras of this monad are called \emph{antiplanar dg (pseudo-)operads}.
\end{defn}

\subsection{Cooperads}

The assumption that $\mathsf V$ is the category of dg vector spaces implies that the functors $\freepl$ and $\freeantipl$ are not only monads, but also in a natural way \emph{comonads}. The structure map
$$ \freepl \to \freepl \circ \freepl$$
is given by ``decomposing trees'' as in \cite[Section 5.8.7]{lodayvallette}. The counit is given by projection onto trees with a single vertex. Coalgebras for the comonad $\freepl$ are \emph{conilpotent planar cooperads}, and coalgebras for $\freeantipl$ are \emph{conilpotent antiplanar cooperads}. All cooperads in this paper will be conilpotent. A \emph{cofree} (anti)planar cooperad is one of the form $\freepl(M)$ (resp. $\freeantipl(M)$) for some planar collection $M$.

\subsection{Cyclic operads}

Cyclic operads are defined just like planar operads, except the construction is built on stable labeled trees, rather than planar stable labeled trees. In particular, stable labeled trees form a category \(\mathsf{T}\), and the free cyclic operad on a collection \(K\) is given functorially by a formula
\[
\freecyc(K)_n =  \mathrm{colim}\bigl(\mathsf{T}_n\xrightarrow{K[\,]}\mathsf{V}\bigr),
\]
exactly as in the planar case, but using the category \(\mathsf{T}\) of (not necessarily planar) stable labeled trees. The only differences are that (i) the free cyclic operad is built summing over a larger class of trees, and (ii) the components of cyclic operads carry actions of symmetric groups.  The free anticyclic operad on $K$ is in the same way given by 
\[
\freeanticyc(K)_n =  \mathrm{colim}\bigl(\mathsf{T}_n\xrightarrow{\mathrm{det}\otimes K[\,]}\mathsf{V}\bigr).
\]
For details on cyclic and anticyclic operads, see \cite{getzlerkapranovcyclic}.

\subsection{Bar and cobar constructions}

Given a collection \(K\), we follow Getzler-Kapranov \cite{getzlerkapranovcyclic} and define its \emph{operadic suspension} \(\Lambda K\) by
\[
\Lambda K_n = \Sigma^{2-n}K_n\otimes \mathrm{sgn}_n.
\]
Here the suspension \(\Sigma\) of a dg vector space is defined by \((\Sigma V)^n = V^{n+1}\). If \(K\) is a cyclic dg (co)operad, then \(\Lambda K\) is an anticyclic dg (co)operad. The same remains true if we replace the adjective (anti)cyclic by (anti)planar. This follows from noting that there is an equality of functors
\[
\freeantipl = \Lambda^{-1} \freepl\Lambda.
\]
Assume that \(\mathsf{O}\) is a planar dg operad. This can be used to define an extra differential \(d_{\barp}\) on the cofree conilpotent antiplanar cooperad \(\freeantipl(\Sigma\mathsf{O})\), in the standard way. It is defined in terms of decorated trees by using the operad compositions
\[
\mathsf{O}[T] \to \bigoplus_{T'=T/e} \mathsf{O}[T']
\]
to contract an edge in all possible ways. (After the suspensions this will square to zero and have degree plus one.) Moreover, \(d_{\barp}\) is a coderivation of the cocompositions of \(\freeantipl(\Sigma\mathsf{O})\). (However, $d_{\barp}$ is not a derivation of the natural \emph{operadic} composition maps of \(\freeantipl(\Sigma\mathsf{O})\), which is why the bar construction must be a cooperad.)

\begin{defn}
	The \emph{planar bar construction} on a planar dg operad \(\mathsf{O}\) is the antiplanar dg cooperad \(\barp\mathsf{O}\) obtained by adding the differential \(d_{\barp}\) to the antiplanar dg cooperad \(\freeantipl(\Sigma\mathsf{O})\).
	
Analogously, if \(\mathsf{O}\) is an antiplanar dg operad one defines the planar bar construction by
	\[
	\barp\mathsf{O} = \bigl( \freepl(\Sigma\mathsf{O}), d_{\barp} \bigr).
	\]
	In this situation the bar construction is a planar dg cooperad.
\end{defn}

Dually, if \(\mathsf{A}\) is a planar dg cooperad, then we get a square-zero, degree \(+1\) derivation \(d_{\cobarp}\) on the free antiplanar operad \(\freepl(\Sigma^{-1}\mathsf{A})\) by summing over all ways to split a vertex into two vertices connected by an edge, using the cocompositions.

\begin{defn}
	The \emph{planar cobar construction} on a planar dg cooperad \(\mathsf{A}\) is the antiplanar dg operad
	\[
	\cobarp\mathsf{A} = \bigl( \freeantipl(\Sigma^{-1}\mathsf{A}), d_{\cobarp} \bigr).
	\]
	If \(\mathsf{A}\) is instead antiplanar, we define the cobar construction
	\[
	\cobarp\mathsf{A} = \bigl( \freepl(\Sigma^{-1}\mathsf{A}), d_{\cobarp} \bigr)
	\]
	as a planar dg operad.
\end{defn}

\begin{defn}
	A morphism $\mathsf O \to \mathsf O'$ of dg (co)operads is a \emph{quasi-isomorphism} if the induced map on cohomology is an isomorphism. If $\mathsf O$ and $\mathsf O'$ are related by a zig-zag of quasi-isomorphisms, then we say that $\mathsf O$ and $\mathsf O'$ are \emph{quasi-isomorphic}.
\end{defn}

\begin{prop}
	The bar and cobar constructions are functorial, related by an adjunction
	\[
	\mathrm{Hom}(\cobarp(\mathsf{A}),\mathsf{O}) = \mathrm{Hom}(\mathsf{A}, \barp(\mathsf{O})),
	\]
	and the natural morphisms
	\[
	\cobarp\,\barp \mathsf{O} \to \mathsf{O}, \qquad \mathsf A \to \barp \, \cobarp \mathsf A,
	\]
	are quasi-isomorphisms of operads and cooperads, respectively. Moreover, if $\mathsf O$ and $\mathsf O'$ are quasi-isomorphic planar operads, then $\barp \mathsf O$ and $\barp \mathsf O'$ are again quasi-isomorphic; and if $\mathsf A\to \mathsf A'$ is a quasi-isomorphism on the associated gradeds of the coradical filtrations, then $\cobarp \mathsf A\to\cobarp \mathsf A'$ is a quasi-isomorphism.
\end{prop}

Everything in this appendix is a specialization to the planar case of theory that is well-known for cyclic operads. In the cyclic case, the bar construction \(\barc \mathsf{O}\) of an anticyclic dg operad \(\mathsf{O}\), for example, is given by adding an edge-contracting differential \(d_{\barc}\) to the free cyclic operad \(\freecyc(\Sigma\mathsf{O})\).

\printbibliography

\end{document}